\documentclass{amsart}
\usepackage{mathrsfs,amssymb,amsmath,amsthm,stmaryrd}
\usepackage[all,cmtip]{xy}
\usepackage{cite}

\newtheorem{theorem}{Theorem}[section]
\newtheorem*{theorem*}{Theorem}
\newtheorem{lemma}[theorem]{Lemma}
\newtheorem{proposition}[theorem]{Proposition}
\newtheorem{corollary}[theorem]{Corollary}

\theoremstyle{definition}
\newtheorem{definition}[theorem]{Definition}
\newtheorem{example}[theorem]{Example}

\theoremstyle{remark}
\newtheorem{remark}[theorem]{\textbf{Remark}}

\numberwithin{equation}{section}

\newcommand{\op}[1]{\operatorname{#1}}
\newcommand{\p}{\prime}

\begin{document}

\title{The asymptotic behavior of the steady gradient K\"ahler-Ricci soliton of the Taub-NUT type of Apostolov and Cifarelli}

\author{Daheng Min}
\address{Mathematisches Institut, University of M\"unster, Germany}
\email{dr.daheng.min@uni-muenster.de}
\date{November 7, 2024}

\begin{abstract}
We first determine the asymptotic cone of the steady gradient K\"ahler-Ricci soliton of the Taub-NUT type constructed by Apostolov and Cifarell in \cite{apostolov2023hamiltonian}. Then we study a special case and prove that it is an ALF Calabi-Yau metric in a certain sense. Finally we construct new ALF Calabi-Yau metrics on crepant resolution of its quotients modeled on it using the method of Tian-Yau-Hein.
\end{abstract}
\maketitle

\section{Introduction}
In \cite{apostolov2023hamiltonian}, families of complete steady gradient K\"ahler-Ricci solitons on $\mathbb{C}^n$ are constructed for $n\geq 2$. In each family, there is a Calabi-Yau metric. In this article, we will mainly consider the family of the Taub-NUT type. In this family, the Calabi-Yau metric is the Taub-NUT metric if $n=2$, and for $n\geq 3$ it is a new example of a complete Calabi-Yau metric. More precisely, it is proved in \cite[Theorem 1.4]{apostolov2023hamiltonian} that, given a partition of the integer $n\geq 2$ as below,
\begin{align}
n = l + \sum_{j=1}^{l-1}d_j, l\geq 2, d_j\geq 0,
\end{align}
there exists an $(l-1)-$dimensional family of non-isometric, irreducible complete steady gradient K\"ahler-Ricci soliton on $\mathbb{C}^n$, all admitting a hamiltonian 2-form of order $l$ and isometry group $U(1)\times\Pi_{j=1}^{l-1}U(d_j)$. One K\"ahler metric $\omega_{l,d_1,\dots,d_{l-1}}$ in each family is a complete Ricci-flat K\"ahler metric on $\mathbb{C}^n$. Moreover, it is proved that the volume growth of the metric is of order $2n-1$.

In this article, we will study the asymptotic cone of the steady gradient K\"ahler-Ricci soliton of the Taub-NUT type of Apostolov and Cifarelli. We will prove the following theorem:
\begin{theorem}\label{theorem introduction 1}
    The asymptotic cone of the K\"ahler-Ricci soliton of the Taub-NUT type of Apostolov and Cifarelli is unique and is $(\prod_{j=1}^{l-1}\mathbb{C}^{d_j+1}/\Lambda)\times \mathbb{R}$. Where $\Lambda$ is a closed subgroup of $\mathbb{T}^{l-1}$ that acts on $\prod_{j=1}^{l-1}\mathbb{C}^{d_j+1}$, and depends on the choice of the parameters.
\end{theorem}
We will explain in detail the construction of Apostolov and CIfarelli in Section \ref{section The steady gradient Kahler-Ricci soliton}, in particular we will introduce the parameters. The group $\Lambda$ will be defined and discussed in Section \ref{section The asymptotic cone of the locally flat metric}. Here we note that the real dimension of $\Lambda$ depends on the parameters and can be $0$ or strictly positive.

Consequently, if we consider the Calabi-Yau metric in the family, then we see that there are many different complete Calabi-Yau metric on $\mathbb{C}^n$ and many of them have different asymptotic cones. This phenomenon also appears in the context of asymptotically conic Calabi-Yau metric, we mention the works \cite{ConlonRochon2021}, \cite{Yangli2017}  and \cite{zbMATH07131296}, which give counter examples to a conjecture of Tian \cite[Remark 5.3]{zbMATH05234298}. In these works, Calabi-Yau metrics on
$\mathbb{C}^n$ with volume growth of order $2n$ and asymptotic cone of the form $V_0\times \mathbb{C}$ are constructed, but in Theorem \ref{theorem introduction 1}, the volume growth is of order $2n-1$.

Another feature of Theorem \ref{theorem introduction 1} is that the asymptotic cone is generally not a smooth cone. Examples of asymptotically conic Calabi-Yau metrics with singular asymptotic cone are studied by Joyce \cite{joyce2000compact} (QALE manifold), Sz\'{e}kelyhidi \cite{zbMATH07131296}, Yang Li \cite{Yangli2017}, Conlon, Degeratu and Rochon \cite{ConlonRochon2019, ConlonRochon2021, conlon2023warped}. Besides the difference of volume growth, another difference between Theorem \ref{theorem introduction 1} and the above works is that in Theorem \ref{theorem introduction 1} the dimension of the asymptotic cone may be strictly smaller than the order of volume growth.

The proof of Theorem \ref{theorem introduction 1} consists of two steps. In the first step, we consider a locally flat metric in Section \ref{section The locally flat metric} and determine its asymptotic cone in Section \ref{section The asymptotic cone of the locally flat metric}. In the second step, we show that the locally flat metric is close to the metric of Apostolov-Cifarelli in a large region, and we prove in Section \ref{section The asymptotic cone of the Kahler-Ricci soliton} that they have the same asymptotic cone.\\

In the special case where $l=2$, the asymptotic cone of the K\"ahler-Ricci soliton is $(\mathbb{C}^{n-1}/\mathbb{Z}_{n-1})\times \mathbb{R}$. And we will show that the Calabi-Yau metric $\omega_{2,n-2}$ in this family is an ALF metric in the following sense.
\begin{theorem}\label{theorem introduction 2}
The Calabi-Yau metric $(\mathbb{C}^n, \omega_{2,n-2}, g_{2,n-2})$ of Apostolov-Cifarelli is an ALF metric in the following sense:
\begin{itemize}
  \item The volume growth of $g_{2,n-2}$ is of order $2n-1$;
  \item The asymptotic cone of $g_{2,n-2}$ is a $(2n-1)$-dimensional metric cone;
  \item The sectional curvature of $g_{2,n-2}$ is bounded by $\frac{C}{\rho}$ for some $C>0$.
\end{itemize}
Here, $\rho$ is the distance function measured by $g_{2,n-2}$ with respect to some point of $\mathbb{C}^n$.
\end{theorem}
This notion of the ALF property has already been considered in \cite{min2023construction}. If we think of the asymptotically conic Calabi-Yau metric as the higher dimensional generalizations of ALE gravitational instantons, then we can regard ALF Calabi-Yau metrics as higher dimensional analogue of ALF gravitational instantons.

In \cite{min2023construction}, many examples of higher dimensional ALF Calabi-Yau metrics of real dimension $4n$ are constructed, and they also have singular asymptotic cones. However, according to Theorem \ref{theorem introduction 2}, there exist ALF Calabi-Yau metrics of any dimension.

The difficulty of the proof of Theorem \ref{theorem introduction 2} is the estimation of curvature. In Section \ref{section The special case}, we apply a result of Naber and Zhang \cite{Naber-Zhang} to show the curvature decay.

Modeled on this ALF Calabi-Yau metric, we can construct new ALF Calabi-Yau metric on the crepant resolution of its quotient. Assume that $\Gamma\subset U(1)\times U(n-1)$ is a finite subgroup such that the singularity $\mathbb{C}^n/\Gamma$ admits a crepant resolution $\pi: Y \rightarrow \mathbb{C}^n/\Gamma$, recall that $(\omega_{2,n-2}, g_{2,n-2})$ is invariant by $U(1)\times U(n-1)$  so it is invariant by $\Gamma$. In Section \ref{section crepant resolution}, we will prove the following theorem using the approach of Tian-Yau's work \cite{zbMATH04186562, zbMATH00059791} and result of of Hein \cite{Heinthesis}, which is a non-compact version of the classical Calabi-Yau theorem.
\begin{theorem}\label{theorem introduction 3}
For any compactly supported K\"ahler class of $Y$ and any $c>0$, there exists an ALF Calabi-Yau metric $\omega^\p$ having the same cohomology class on $Y$ which is asymptotic to $c\omega_{2,n-2}$ near the infinity. More precisely, we have
\begin{align}
|\nabla^k(\omega^\p - c\pi^*\omega_{2,n-2})|_{\omega^\p} \leq C(k,\epsilon)(1 + \rho^\p)^{-2n+3+\epsilon},
\end{align}
where $\epsilon > 0$ is any small constant, $\rho^\p$ is the distance function from a point of $Y$ measured by $\omega^\p$ and $k\geq 0$.
\end{theorem}
In the work of Van Coevering \cite{VanCoevering}, asymptotically conic Calabi-Yau metrics are constructed in each compactly supported K\"ahler class of crepant resolution of Ricci-flat K\"ahler cone. So Theorem \ref{theorem introduction 3} could be understood as an ALF analogue of the result of \cite{VanCoevering}.

From another point of view, in the case $n=1$, there is a ``Kummer construction'' of ALF-$D_k$ instantons discussed in the work of Biquard and Minerbe \cite{biquard2011kummer}. So we can also think of Theorem \ref{theorem introduction 3} as a higher dimensional analogue of \cite{biquard2011kummer}.

As an application of Theorem \ref{theorem introduction 3}, consider the crepant resolution $\mathcal{K}_{\mathbb{CP}^{n-1}} \rightarrow \mathbb{C}^n/\mathbb{Z}_n$, we have
\begin{corollary}
    There exist ALF Calabi-Yau metrics on $\mathcal{K}_{\mathbb{CP}^{n-1}}$ asymptotic to $(\omega_{2,n-2},g_{2,n-2})/\mathbb{Z}_n$ with asymptotic cone $(\mathbb{C}^{n-1}/\mathbb{Z}_{k(n-1)})\times \mathbb{R}$, where $k = n$ if $n$ is odd and $k=\frac{n}{2}$ if $n$ is even. Here $\mathcal{K}_{\mathbb{CP}^{n-1}}$ is the total space of the canonical bundle of $\mathbb{CP}^{n-1}$.
\end{corollary}

Finally, let us discuss some aspects that are still open. In \cite{zbMATH07698520}, it is proved that for a complete Calabi-Yau metric of maximal voulume growth, the quadratic curvature decay ($|\op{Rm}|\leq \frac{C}{\rho^2}$) is equivalent to being asymptotically conic. Then is natural to ask whether there is a similar result for non-maximal volume growth. For example, for ALF Calabi-Yau metrics, is there a relation between the quadratic curvature decay and the smoothness of the asymptotic cone? In the four dimensional case, we know that ALF gravitational instantons have faster than quadratic curvature decay and smooth asymptotic cone. But in higher dimensions, we have no examples of (non-trivial) ALF Calabi-Yau metrics with quadratic curvature decay or smooth asymptotic cone, and it will be interesting to find such examples.

In the recent work \cite{cifarelli2024explicitcompletericciflatmetrics} of Cifareli, more complete Calabi-Yau metrics and K\"ahler-Ricci solitons are constructed, generalizing \cite{apostolov2023hamiltonian}. Then it will be interesting to understand the asymptotic cones of these new examples.

\section{The steady gradient K\"ahler-Ricci soliton}\label{section The steady gradient Kahler-Ricci soliton}
In this section we will give a description of the steady gradient K\"ahler-Ricci soliton of the Taub-NUT type of Apostolov and Cifarelli following \cite{apostolov2023hamiltonian}. First we fix $n\geq 2$ and a partition of $n$:
\begin{align}
n = l + \sum_{j=1}^{l-1}d_j, l\geq 2, d_j\geq 0.
\end{align}
We also fix $l$ real numbers $\alpha_1,\dots,\alpha_l$ such that $\alpha_1<\alpha_2<\dots<\alpha_l$, and define $\mathring{D} = (-\infty,\alpha_1)\times(\alpha_1,\alpha_2)\times\dots\times(\alpha_{l-2},\alpha_{l-1})\times(\alpha_l,+\infty)$. Note that there is a gap between the last two intervals. Later we will use $(\xi_1,\dots,\xi_l)\in \mathring{D}$ as its coordinates.

Define
\begin{align}
    p_c(t) = \prod_{j=1}^{l-1}(t-\alpha_j)^{d_j},\\
    p_{nc}(t) = \prod_{j=1}^l(t-\xi_j).
\end{align}
They are polynomials in $t$ of degree $n-l$ and $l$. Observe that for $j=1,2,\dots,l-1$, $(-1)^{l-j}p_{nc}(\alpha_j)>0$.
\begin{remark}
In terms of the hamiltonian 2-form $\phi$, the real numbers $\alpha_1,\dots,\alpha_{l-1}$ are the constant roots of the momentum polynomial $p(t) = (-1)^n\operatorname{pf}(\phi-t\omega)$ with multiplicities $d_1,\dots,d_{l-1}$, while $\xi_1,\dots,\xi_l$ are the non-constant roots of $p(t)$. So we have $p(t) = p_c(t)p_{nc}(t)$.
\end{remark}
Fix a real number $a\in \mathbb{R}$, let
\begin{align}
    P(t) &= \prod_{j=1}^{l-1}(t-\alpha_j)^{d_j+1},\\
    q(t) &= \frac{P^\prime(t) + 2aP(t)}{p_c(t)} = \sum_{j=1}^{l-1}(2a(t-\alpha_j)+(d_j+1))\prod_{k=1,k\neq j}^{l-1}(t-\alpha_k).
\end{align}
They are polynomials in $t$ of degree $n-1$ and $l-1$ ($l-2$ if $a=0$). Observe that for $j=1,2,\dots,l-1$, we have $q(\alpha_j) = (d_j + 1)\prod_{k=1,k\neq j}^{l-1}(\alpha_j-\alpha_k)$, so $(-1)^{l-1-j}q(\alpha_j) >0$.
We also set $F_1(t) = \dots F_{l-1}(t) = P(t)$, $F_l(t) = P(t) - e^{2a(\alpha_l - t)} P(\alpha_l)$.
Note that $F_l(t) >0$ for $t>\alpha_l$.

For $j=1,\dots,l-1$, define $(\check{g}_j^0, \check{\omega}_j^0)$ as the Fubini-Study metric on $\mathbb{CP}^{d_j}$ of constant scalar curvature $2d_j(d_j+1)$, so that $[\check{\omega}_j^0]\in H^2(\mathbb{CP}^{d_j}, \mathbb{Z})$ is the primitive generator. Define $(g_j, \omega_j) = \frac{2(d_j + 1)}{(-1)^{l-1-j}q(\alpha_j)}(\check{g}_j^0, \check{\omega}_j^0)$, so it is a K\"ahler-Einstein metric of constant scalar curvature $(-1)^{l-1-j}d_jq(\alpha_j)$. Formally, we may set $d_l=0$ so that $\mathbb{CP}^{d_l}$ is a point, and $\prod_{j=1}^l\mathbb{CP}^{d_j} = \prod_{j=1}^{l-1}\mathbb{CP}^{d_j}$. We can also define $(\check{g}_l^0, \check{\omega}_l^0)$ and $(g_j, \omega_j)$ using the same formula, and we should think of them as zero tensors in the product $\prod_{j=1}^l\mathbb{CP}^{d_j}$.

For $j=1,\dots,l$, define
\begin{align}
    v_j = (-1)^{l-j}\frac{2(d_j+1)}{q(\alpha_j)}(-\alpha_j^{l-1},\dots,(-1)^r\alpha_j^{l-r},\dots,(-1)^l)\in \mathbb{R}^l
\end{align}
As $\alpha_j$ are distinct, $(v_1,\dots,v_l)$ form a basis of $\mathbb{R}^l$ by the Vandermonde determinant. Let $\Gamma_v$ be the lattice generated by $(v_1,\dots,v_l)$, then $\mathbb{T}^l = \mathbb{R}^l/\Gamma_v$ is a $l$-dimensional torus. Define $P$ as the $\mathbb{T}^l$-principal bundle over $\prod_{j=1}^l\mathbb{CP}^{d_j}$ with connection $1$-form $\theta$ such that
\begin{align}
d\theta = \sum_{j=1}^l\check{\omega}_j^0\otimes v_j.
\end{align}
More precisely, if we write $\theta = (\theta_1,\dots,\theta_l)$, then for $r = 1,\dots,l$, we have
\begin{align}\label{equation d theta}
    d\theta_r = \sum_{j=1}^{l-1}(-1)^{l-j+r}\frac{2}{\prod_{k=1,k\neq j}^{l-1}(\alpha_j-\alpha_k)}\alpha_j^{l-r}\check{\omega}_j^0.
\end{align}

\begin{remark}
Here we recall that $\check{\omega}_j^0$ is the primitive generator in $H^2(\mathbb{CP}^{d_j},\mathbb{Z})$, and $v_j$ is the generator of $\Gamma_v$, so the curvature $2$-form $d\theta$ is indeed integral. In fact, $P$ is diffeomorphic to the $\mathbb{T}^l$-principal bundle over $\prod_{j=1}^l\mathbb{CP}^{d_j}$ corresponding to $\bigoplus_{j=1}^lO_{\mathbb{CP}^{d_j}}(-1)$.
\end{remark}

Let $M^0 = \mathring{D}\times P$, then we define the following metric and 2-form on $M^0$:
\begin{align}\label{equation of g}
    g =& \sum_{j=1}^{l-1}(-1)^{l-j}p_{nc}(\alpha_j)\check{g}_j + \sum_{j=1}^l\frac{p_c(\xi_j)\Delta(\xi_j)}{F_j(\xi_j)}(d\xi_j)^2 \\
    &+ \sum_{j=1}^l\frac{F_j(\xi_j)}{p_c(\xi_j)\Delta(\xi_j)}\left(\sum_{r=1}^l\sigma_{r-1}(\hat{\xi}_j)\theta_r \right)^2,\\
    \omega =& \sum_{j=1}^{l-1}(-1)^{l-j}p_{nc}(\alpha_j)\check{\omega}_j + \sum_{r=1}^l d\sigma_r\wedge\theta_r.
\end{align}
Here, $\Delta(\xi_j) = \prod_{i=1,i\neq j}^l(\xi_j-\xi_i)$, $\sigma_0=1,\sigma_1,\dots,\sigma_l$ are elementary symmetric polynomials of $\xi_1,\dots,\xi_l$ so $p_{nc}(t) = \sum_{r=0}^l (-1)^r t^{l-r}\sigma_r$, and $\sigma_{r-1}(\hat{\xi}_j)$ is the $(r-1)$-th elementary symmetric polynomial of $\{\xi_i|i\neq j\}$.

In \cite{ACGI} and \cite{apostolov2023hamiltonian} it is shown that $(M^0,g,\omega)$ is a steady gradient K\"ahler-Ricci soliton with complex structure $J$ given by
\begin{align}
    Jd\xi_j &= \frac{F_j(\xi_j)}{p_c(\xi_j)\Delta(\xi_j)}\left(\sum_{r=1}^l\sigma_{r-1}(\hat{\xi}_j)\theta_r \right),\\
    J\theta_r &= (-1)^r\sum_{j=1}^l\frac{p_c(\xi_j)}{F_j(\xi_j)}\xi_j^{l-r}d\xi_j.
\end{align}
and the soliton vector field $X$ have Killing potential $a\sigma_1$. To better understand the soliton vector field, we discuss as follows.

As noted in \cite{apostolov2023hamiltonian}, $M^0$ is diffeomorphic to the $(\mathbb{C}^*)^l$-principal bundle $(\mathbb{C}^*)^l\times_{\mathbb{T}^l}P$ corresponding to the split vector bundle $\hat{M} = \bigoplus_{j=1}^lO_{\mathbb{CP}^{d_j}}(-1)$. So we can think of $M^0$ as a dense subset of $\hat{M}$. For $j=1,\dots,l$, denote by $T_j$ the generator of rotation in each component $O_{\mathbb{CP}^{d_j}}(-1)$ of $\hat{M}$. If we identify $\mathbb{R}^l$ with the Lie algebra of $\mathbb{T}^l = \mathbb{R}^l/\Gamma_v$, then $v_j$ corresponds to $T_j$ by our construction of $P$. Denote by $e_1,\dots,e_l$ the canonical basis of $\mathbb{R}^l$, and let $K_1,\dots,K_l$ be the corresponding vector field on $\hat{M}$, then $(K_1,\dots,K_l)$ is dual to $(\theta_1,\dots,\theta_l)$. It follows that the moment map of $K_j$ with respect to $\omega$ is $\sigma_j$. In particular, the soliton vector field is given by $aK_1$. In terms of $T_1,\dots,T_l$, by inverting the Vandermonde matrix, we find that
\begin{align}\label{equation expand e_1 in v_j}
    e_1 = \sum_{j=1}^l(-1)^{l+1-j}\frac{q(\alpha_j)}{2(d_j + 1)}\frac{1}{\prod_{k=1,k\neq j}^l(\alpha_j-\alpha_k)}v_j.
\end{align}
It follows that
\begin{align}\label{equation K_1 in terms of T_j}
    K_1 = \sum_{j=1}^l(-1)^{l+1-j}\frac{q(\alpha_j)}{2(d_j + 1)}\frac{1}{\prod_{k=1,k\neq j}^l(\alpha_j-\alpha_k)}T_j,
\end{align}
and
\begin{align}
    X = a\sum_{j=1}^l(-1)^{l+1-j}\frac{q(\alpha_j)}{2(d_j + 1)}\frac{1}{\prod_{k=1,k\neq j}^l(\alpha_j-\alpha_k)}T_j.
\end{align}

Observe that there is a $\mathbb{T}^{d_j}$-symmetry on $\mathbb{CP}^{d_j}$, combining these $\mathbb{T}^{d_j}$-actions with the $\mathbb{T}^l$-action, we get a $\mathbb{T}^n$-action on $M^0$ and, in fact, $(M^0,g,\omega)$ is $\mathbb{T}^n$ invariant.

Under the blow-down map $\hat{M}\rightarrow \mathbb{C}^n$, we can view $M^0$ as a dense subset of $\mathbb{C}^n$. In \cite{apostolov2023hamiltonian} it is proved that $(g,\omega)$ defined on $M^0$ extends to a smooth $\mathbb{T}^n$-invariant steady gradient K\"ahler-Ricci soliton on $\mathbb{R}^{2n} = \prod_{j=1}^l\mathbb{R}^{2(d_j+1)}$ compatible with the standard symplectic form and with the soliton vector field
\begin{align}
    X = a\sum_{j=1}^l(-1)^{l+1-j}\frac{q(\alpha_j)}{2(d_j + 1)}\frac{1}{\prod_{k=1,k\neq j}^l(\alpha_j-\alpha_k)}X_j,
\end{align}
where $X_j$ is the vector field on $\mathbb{R}^{2(d_j+1)}\cong \mathbb{C}^{d_j+1}$ with flow the multiplication with $e^{2\pi it}$. If, furthermore $a\geq 0$, then the complex structure is $\mathbb{T}^n$-equivariantly biholomorphic to the standard complex structure on $\mathbb{C}^n$ and the metric is complete. In particular, if $a=0$, then we get a complete Calabi-Yau metric $\omega_{l,d_1,\dots,d_{l-1}}$ on $\mathbb{C}^n$.
\begin{remark}
    As we shall see later, the expansion coefficients of $X$ in terms of $X_j$ play an important role in determining the asymptotic cone of $(\mathbb{C}^n, g, \omega)$.
\end{remark}

Regarding the parameters, once the discrete parameters $l,d_1,\dots,d_{l-1}$ are fixed, the above construction depends on the following $l+1$ continuous parameters $a$ and $\alpha_1,\dots,\alpha_l$. As pointed out in \cite{apostolov2023hamiltonian}, for $c>0, d\in \mathbb{R}$, the data $(a,\alpha_1,\dots,\alpha_l)$ define the same steady gradient K\"ahler-Ricci soliton as $(\frac{a}{c},c\alpha_1+d,\dots,c\alpha_l+d)$. It is then possible to normalize $\alpha_1=0,\alpha_2=1$ and it is shown that different choices of $(a,\alpha_3,\dots,\alpha_l)$ give non-isometric K\"ahler metrics.

\section{The locally flat metric}\label{section The locally flat metric}
Recall that in formula \eqref{equation of g}, we have $F_1(t)=\dots=F_{l-1}(t)=P(t)$, but $F_{l}(t) = P(t) - e^{2a(\alpha_l - t)} P(\alpha_l)$. If we replace $F_l(t)$ by $P(t)$, then we will get the following metric on $M^0$:
\begin{align}\label{equation of g prime}
    g^\p =& \sum_{j=1}^{l-1}(-1)^{l-j}p_{nc}(\alpha_j)\check{g}_j + \sum_{j=1}^l\frac{p_c(\xi_j)\Delta(\xi_j)}{P(\xi_j)}(d\xi_j)^2 \\
    &+ \sum_{j=1}^l\frac{P(\xi_j)}{p_c(\xi_j)\Delta(\xi_j)}\left(\sum_{r=1}^l\sigma_{r-1}(\hat{\xi}_j)\theta_r \right)^2\\
      =& \sum_{j=1}^{l-1}(-1)^{l-j}p_{nc}(\alpha_j)\check{g}_j + \sum_{j=1}^l\frac{\Delta(\xi_j)}{\prod_{k=1}^l(\xi_j-\alpha_k)}(d\xi_j)^2 \\
      &+ \sum_{j=1}^l\frac{\prod_{k=1}^l(\xi_j-\alpha_k)}{\Delta(\xi_j)}\left(\sum_{r=1}^l\sigma_{r-1}(\hat{\xi}_j)\theta_r \right)^2.
\end{align}
According to \cite[Proposition 17]{ACGI}, one can deduce that $g^\p$ defines a locally flat metric. The aim of this section is to provide another proof of this fact, which will reveal the global behavior of $g^\p$.

Define
\begin{align}
    \check{g}^\p &= \sum_{j=1}^{l-1}(-1)^{l-j}p_{nc}(\alpha_j)\check{g}_j,\\
    g^\p_\xi &= \sum_{j=1}^l\frac{\Delta(\xi_j)}{\prod_{k=1}^l(\xi_j-\alpha_k)}(d\xi_j)^2,\\
    g^\p_\theta &= \sum_{j=1}^l\frac{\prod_{k=1}^l(\xi_j-\alpha_k)}{\Delta(\xi_j)}\left(\sum_{r=1}^l\sigma_{r-1}(\hat{\xi}_j)\theta_r \right)^2,
\end{align}
then $g^\p = \check{g}^\p + g^\p_\xi + g^\p_\theta$. The strategy of the proof is simply a change of variable. We will change the coordinates from $(\xi_1,\dots,\xi_l)$ to $(\sigma_1, p_{nc}(\alpha_1),\dots,p_{nc}(\alpha_{l-1}))$.

\subsection{Change of variable for $g^\p_{\theta}$}
Define
\begin{align}
\beta_0 = \sum_{r=1}^l\sigma_{r-1}(\alpha_1,\dots,\alpha_{l-1})\theta_r,
\end{align}
Here $\sigma_{r-1}(\alpha_1,\dots,\alpha_{l-1})$ is the $(r-1)$-th elementary symmetric polynomial of $\alpha_1,\dots,\alpha_{l-1}$. Recall that we have already used $\sigma_1,\dots,\sigma_r$ to denote the elementary symmetric polynomials of $\xi_1,\dots,\xi_r$, when the variables are not $\xi_i$, we will indicate them explicitly.
\begin{lemma}\label{lemma d beta_0}
    We have $d\beta_0 = 0$.
\end{lemma}
\begin{proof}
    Note that
    \begin{align}
        \sum_{r=1}^l\sigma_{r-1}(\alpha_1,\dots,\alpha_{l-1})t^{l-r}=\prod_{j=1}^{l-1}(t + \alpha_j)
    \end{align}
    has roots $-\alpha_1,\dots,-\alpha_{l-1}$. Combining this with equation \eqref{equation d theta} will finish the proof.
\end{proof}
For $i=1,\dots,l-1$, define
\begin{align}
    \beta_{i} = \sum_{r=2}^l\sigma_{r-2}(\alpha_1,\dots,\hat{\alpha}_i,\dots \alpha_{l-1})\theta_r.
\end{align}
Where $\sigma_{r-2}(\alpha_1,\dots,\hat{\alpha}_i,\dots \alpha_{l-1})$ denotes the $(r-2)$-th elementary polynomial of $\{\alpha_k|1\leq k \leq l-1, k\neq i\}$.
\begin{lemma}\label{lemma d beta_i}
    For $i=1,\dots,l-1$, we have $d\beta_i = (-1)^{l-i}2\check{\omega}_i^0$.
\end{lemma}
\begin{proof}
    One applies the formula of the inverse of the Vandermonde matrix. More precisely, consider the $(l-1)\times(l-1)$ vandermonde matrix $((-1)^r\alpha_j^{l-r})_{1\leq j\leq l-1, 2\leq r\leq l}$, here $j$ is the row number and $r$ is the colume number. Its inverse $(\lambda_{ri})_{2\leq r\leq r, 1\leq i \leq l-1}$ is given by
    \begin{align}
        \lambda_{ri} = \frac{\sigma_{r-2}(\alpha_1,\dots,\hat{\alpha}_i,\dots \alpha_{l-1})}{\Delta_i(\alpha_1,\dots,\alpha_{l-1})},
    \end{align}
    where $\Delta_i(\alpha_1,\dots,\alpha_{l-1}) = \prod_{k=1,k\neq i}^{l-1}(\alpha_i - \alpha_k)$. Once again, combining this with equation \eqref{equation d theta} will finish the proof.
\end{proof}
The main result of this subsection is the following.
\begin{proposition}\label{proposition g prime theta}
    We have
    \begin{align}\label{equation g prime theta}
        g^\p_{\theta} = \beta_0^2 + \sum_{j=1}^{l-1}\frac{-p_{nc}(\alpha_j)}{\Delta_j(\alpha_1,\dots,\alpha_{l-1})}\beta_j^2.
    \end{align}
\end{proposition}
\begin{proof}
    We will show that, evaluated by a basis $Y_1,\dots,Y_l$ of the linear space generated by $K_1,\dots,K_l$, both sides of the above equation are the same. To simplify the notation, let the right-hand side be denoted by $g^\p_\beta$.

    For $j=1,\dots,l$, define
    \begin{align}
        Y_j = \sum_{r=1}^l(-1)^r\alpha_j^{l-r}K_r,
    \end{align}
    so $Y_1,\dots,Y_l$ generates the same space as $K_1,\dots,K_l$, which is the vertical direction of the $\mathbb{T}^l$-principal bundle $P$. In fact, we have $T_j = (-1)^{l-j}\frac{2(d_j+1)}{q(\alpha_j)}Y_j$.

    Recall that $K_1,\dots,K_r$ is the dual basis of $\theta_1,\dots,\theta_l$, then by the same argument as in the proof of Lemma \ref{lemma d beta_0}, we have
    \begin{align}
        \beta_0(Y_1) =\dots =\beta_0(Y_{l-1}) = 0,\\
        \beta_0(Y_l) = -\prod_{k=1}^{l-1}(\alpha_l-\alpha_k).
    \end{align}
    Similarly, for $1\leq i, j \leq l-1$, we have
    \begin{align}
        \beta_i(Y_j) = \delta_{ij}\Delta_i(\alpha_1,\dots,\alpha_{l-1}),\\
        \beta_i(Y_l) = \prod_{k=1,k\neq i}^{l-1}(\alpha_l - \alpha_k).
    \end{align}
    So, up to some nonzero constant coefficients, the dual basis of $(\beta_0,\beta_1,\dots,\beta_{l-1})$ is $(K_1,T_1,\dots,T_{l-1})$. For $i=1,\dots,l$, we also have
    \begin{align}\label{equation theta component of g prime}
        \left( \sum_{r=1}^l\sigma_{r-1}(\hat{\xi}_i)\theta_r \right)(Y_j) = \sum_{r=1}^l\sigma_{r-1}(\hat{\xi}_i)(-1)^r\alpha_j^{l-r} =
        -\prod_{k=1,k\neq i}^{l}(\alpha_j-\xi_k).
    \end{align}
    It follows that for $1\leq m,n \leq l$, we have
    \begin{align}
        g_\theta^\p(Y_m,Y_n) = \sum_{i=1}^l\frac{\prod_{k=1}^{l-1}(\xi_i - \alpha_k)}{\Delta(\xi_i)}\prod_{k=1,k\neq i}^{l}(\alpha_m-\xi_k)\prod_{k=1,k\neq i}^{l}(\alpha_n-\xi_k).
    \end{align}

    First we consider the case where $1\leq m,n \leq l-1$ and $m\neq n$, then we have
    \begin{align}
        g_\theta^\p(Y_m,Y_n) = p_{nc}(\alpha_m)p_{nc}(\alpha_n)\sum_{i=1}^l\frac{\prod_{k=1,k\neq m,n}^{l-1}(\xi_i - \alpha_k)}{\Delta(\xi_i)}.
    \end{align}
    Note that $\prod_{k=1,k\neq m,n}^{l-1}(\xi_i - \alpha_k)$ can be viewed as a polynomial in $\xi_i$ of degree $l-3$, so by the following Vandermonde identity for $s=1,\dots,l$,
    \begin{align}\label{equation vandermonde identity basic}
        \sum_{i=1}^l\frac{\xi_i^{l-s}}{\Delta(\xi_i)} = \delta_{s1},
    \end{align}
    we conclude that in this case $g_\theta^\p(Y_m,Y_n) = 0$. It is also clear that in this case we have $g_\beta^\p(Y_m,Y_n) = 0$, so $g_\theta^\p(Y_m,Y_n) = g_\beta^\p(Y_m,Y_n)$.

    Next, we consider the case $m=n$ and $1\leq m \leq l-1$. For the left-hand side we have
    \begin{align}
        g_\theta^\p(Y_m,Y_m) = p_{nc}(\alpha_m)^2\sum_{i=1}^l\frac{\prod_{k=1,k\neq m}^{l-1}(\xi_i - \alpha_k)}{\Delta(\xi_i)(\xi_i-\alpha_m)}.
    \end{align}
    So we are led to consider $\sum_{i=1}^l\frac{\xi_i^s}{\Delta(\xi_i)(\xi_i-\alpha)}$ for $0\leq s \leq l-2$. As a starting point, by Lagrange interpolation, we have
    \begin{align}
        -p_{nc}(\alpha)\sum_{i=1}^l\frac{1}{\Delta(\xi_i)(\xi_i-\alpha)} = 1,
    \end{align}
    here in the above we have a polynomial in $\alpha$ of degree at most $l-1$ which equals $1$ at points $\xi_1,\dots,\xi_l$. So we have
    \begin{align}
        \sum_{i=1}^l\frac{1}{\Delta(\xi_i)(\xi_i-\alpha)} = -\frac{1}{p_{nc}(\alpha)}.
    \end{align}
    Observe that $\frac{\xi_i^s}{\xi_i-\alpha} = \xi_i^{s-1} + \alpha\frac{\xi_i^{s-1}}{\xi_i-\alpha}$. So by induction and Vandermonde identity \eqref{equation vandermonde identity basic}, for $s = 0,\dots,l-1$, we have
    \begin{align}\label{equation vandermonde identity complicated}
        \sum_{i=1}^l\frac{\xi_i^s}{\Delta(\xi_i)(\xi_i-\alpha)} = -\frac{\alpha^s}{p_{nc}(\alpha)}.
    \end{align}
    Since
    \begin{align}
        \prod_{k=1,k\neq m}^{l-1}(\xi_i - \alpha_k) = \sum_{p=0}^{l-2}(-1)^p\sigma_p(\alpha_1,\dots,\hat{\alpha}_m,\dots,\alpha_{l-1})\xi_i^{l-2-p},
    \end{align}
    we deduce that
    \begin{align}
        g_\theta^\p(Y_m,Y_m) &= p_{nc}(\alpha_m)^2\sum_{i=1}^l\frac{\prod_{k=1,k\neq m}^{l-1}(\xi_i - \alpha_k)}{\Delta(\xi_i)(\xi_i-\alpha_m)}\\
        &= p_{nc}(\alpha_m)^2\sum_{p=0}^{l-2}(-1)^p\sigma_p(\alpha_1,\dots,\hat{\alpha}_m,\dots,\alpha_{l-1})\prod_{i=1}^l\frac{\xi_i^{l-2-p}}{\Delta(\xi_i)(\xi_i-\alpha_m)}\\
        &= -p_{nc}(\alpha_m)^2\sum_{p=0}^{l-2}(-1)^p\sigma_p(\alpha_1,\dots,\hat{\alpha}_m,\dots,\alpha_{l-1})\frac{\alpha_m^{l-2-p}}{p_{nc}(\alpha_m)}\\
        &= -p_{nc}(\alpha_m)\sum_{p=0}^{l-2}(-1)^p\sigma_p(\alpha_1,\dots,\hat{\alpha}_m,\dots,\alpha_{l-1})\alpha_m^{l-2-p}\\
        &= -p_{nc}(\alpha_m)\Delta_m(\alpha_1,\dots,\alpha_{l-1}).
    \end{align}
    At the same time, it is easy to verify that
    \begin{align}
        g_\beta^\p(Y_m,Y_m) = -p_{nc}(\alpha_m)\Delta_m(\alpha_1,\dots,\alpha_{l-1})
    \end{align}
    So in this case, we have $g_\theta^\p(Y_m,Y_m) = g_\beta^\p(Y_m,Y_m)$.

    Then for $1\leq m \leq l-1$, we compare $g_\theta^\p(Y_m,Y_l)$ and $g_\beta^\p(Y_m,Y_l)$. For $g_\theta^\p(Y_m,Y_l)$ we have
    \begin{align}
        g_\theta^\p(Y_m,Y_l) &= p_{nc}(\alpha_m)p_{nc}(\alpha_l)\sum_{i=1}^l\frac{\prod_{k=1,k\neq m}^{l-1}(\xi_i - \alpha_k)}{\Delta(\xi_i)(\xi_i - \alpha_l)}\\
        &= p_{nc}(\alpha_m)p_{nc}(\alpha_l)\sum_{p=0}^{l-2}(-1)^p\sigma_p(\alpha_1,\dots,\hat{\alpha}_m,\dots,\alpha_{l-1})\sum_{i=1}^l\frac{\xi_i^{l-2-p}}{\Delta(\xi_i)(\xi_i-\alpha_l)}\\
        &= -p_{nc}(\alpha_m)p_{nc}(\alpha_l)\sum_{p=0}^{l-2}(-1)^p\sigma_p(\alpha_1,\dots,\hat{\alpha}_m,\dots,\alpha_{l-1})\frac{\alpha_l^{l-2-p}}{p_{nc}(\alpha_l)}\\
        &= -p_{nc}(\alpha_m)\prod_{k=1,k\neq m}^{l-1}(\alpha_l-\alpha_k).
    \end{align}
    And it is easy to verify that
    \begin{align}
        g_\beta^\p(Y_m,Y_l) = -p_{nc}(\alpha_m)\prod_{k=1,k\neq m}^{l-1}(\alpha_l-\alpha_k).
    \end{align}
    So $g_\theta^\p(Y_m,Y_l) = g_\beta^\p(Y_m,Y_l)$.

    Finally, we compare $g_\theta^\p(Y_l,Y_l)$ and $g_\beta^\p(Y_l,Y_l)$. Taking the derivative with respect to $\alpha$ in equation \eqref{equation vandermonde identity complicated}, for $s = 0,1,\dots,l-1$ we get
    \begin{align}
        \sum_{i=1}^l\frac{\xi_i^s}{\Delta(\xi_i)(\xi_i - \alpha)^2} = \frac{p_{nc}^\p(\alpha)\alpha^s}{p_{nc}(\alpha)^2} - \frac{s\alpha^{s-1}}{p_{nc}(\alpha)}.
    \end{align}
    So we have
    \begin{align}
        g_\theta^\p(Y_l,Y_l)
        &= p_{nc}(\alpha_l)^2\sum_{i=1}^l\frac{\prod_{k=1}^{l-1}(\xi_i - \alpha_k)}{\Delta(\xi_i)(\xi_i - \alpha_l)^2}\\
        &= p_{nc}(\alpha_l)^2 \sum_{p=0}^{l-1}(-1)^p\sigma_p(\alpha_1,\dots,\alpha_{l-1}) \sum_{i=1}^l\frac{\xi_i^{l-1-p}}{\Delta(\xi_i)(\xi_i - \alpha_l)^2}\\
        &=  p_{nc}(\alpha_l)^2 \sum_{p=0}^{l-1}(-1)^p\sigma_p(\alpha_1,\dots,\alpha_{l-1}) \frac{p_{nc}^\p(\alpha_l)\alpha_l^{l-1-p} - p_{nc}(\alpha_l)(l-1-p)\alpha_l^{l-2-p}}{p_{nc}(\alpha_l)^2}\\
        &= p_{nc}^\p(\alpha_l)\prod_{k=1}^{l-1}(\alpha_l-\alpha_k) - p_{nc}(\alpha_l)\frac{d}{d\alpha_l}\prod_{k=1}^{l-1}(\alpha_l-\alpha_k)
    \end{align}
    It follows that
    \begin{align}
        \frac{g_\theta^\p(Y_l,Y_l)}{(\prod_{k=1}^{l-1}(\alpha_l-\alpha_k))^2}
        = \frac{d}{d\alpha_l}\left(\frac{p_{nc}(\alpha_l)}{\prod_{k=1}^{l-1}(\alpha_l-\alpha_k)} \right).
    \end{align}
    As for $g_\beta^\p(Y_l,Y_l)$, we have
    \begin{align}
        g_\beta^\p(Y_l,Y_l)
        &= \left(\prod_{k=1}^{l-1}(\alpha_l-\alpha_k)\right)^2 + \sum_{i=1}^{l-1}\frac{-p_{nc}(\alpha_i)}{\Delta_i(\alpha_1,\dots,\alpha_{l-1})}\prod_{k=1,k\neq i}^{l-1}(\alpha_l - \alpha_k)^2.
    \end{align}
    So we have
    \begin{align}
        \frac{g_\beta^\p(Y_l,Y_l)}{(\prod_{k=1}^{l-1}(\alpha_l-\alpha_k))^2}
        &= 1 + \sum_{i=1}^{l-1}\frac{-p_{nc}(\alpha_i)}{\Delta_i(\alpha_1,\dots,\alpha_{l-1})(\alpha_i-\alpha_l)^2}\\
        &= 1 + \frac{d}{d\alpha_l}\sum_{i=1}^{l-1}\frac{p_{nc}(\alpha_i)}{\Delta_i(\alpha_1,\dots,\alpha_{l-1})(\alpha_i-\alpha_l)}.
    \end{align}
    Combining the following extended Vandermonde identity
    \begin{align}\label{equation vandermonde identity extended}
        \sum_{j=1}^{l}\frac{\xi_j^{l-1+p}}{\Delta(\xi_j)} = h_p,
    \end{align}
    where $p\geq 0$ and $h_p$ is the $p-$th complete symmetric function of $\xi_1,\dots,\xi_l$, and a similar induction on equation \eqref{equation vandermonde identity complicated}, one obtains the following.
    \begin{align}\label{equation vandermonde identity complicated extended}
        \sum_{i=1}^l\frac{\xi_i^{l+p}}{\Delta(\xi_i)(\xi_i-\alpha)}
        = \sum_{k=0}^ph_{p-k}\alpha^k - \frac{\alpha^{l+p}}{p_{nc}(\alpha)}.
    \end{align}
    Applying \eqref{equation vandermonde identity complicated extended}, we have
    \begin{align}\label{equation beginning of p_nc alpha_l in terms of p_nc alpha_j}
        \sum_{i=1}^{l-1}\frac{p_{nc}(\alpha_i)}{\Delta_i(\alpha_1,\dots,\alpha_{l-1})(\alpha_i-\alpha_l)}
       &= \sum_{p=0}^l(-1)^p\sigma_p\sum_{i=1}^{l-1}\frac{\alpha_i^{l-p}}{\Delta_i(\alpha_1,\dots,\alpha_{l-1})(\alpha_l-\alpha_i)}\\
       &= \sum_{p=0}^l(-1)^p\sigma_p\frac{\alpha_l^{l-p}}{\prod_{k=1}^{l-1}(\alpha_l-\alpha_k)} + \sigma_1 - h_1(\alpha_1,\dots,\alpha_{l-1}) -\alpha_l\\
       &= \frac{p_{nc}(\alpha_l)}{\prod_{k=1}^{l-1}(\alpha_l-\alpha_k)} + \sigma_1 - h_1(\alpha_1,\dots,\alpha_{l-1}) -\alpha_l.\label{equation ending of p_nc alpha_l in terms of p_nc alpha_j}
    \end{align}
    Note that $\sigma_1$ and $h_1(\alpha_1,\dots,\alpha_{l-1})$ are independent of $\alpha_l$, so we get
    \begin{align}
        g_\beta^\p(Y_l,Y_l) = \frac{d}{d\alpha_l}\left(\frac{p_{nc}(\alpha_l)}{\prod_{k=1}^{l-1}(\alpha_l-\alpha_k)} \right) = g_\theta^\p(Y_l,Y_l),
    \end{align}
    which completes the proof.
\end{proof}
\begin{remark}
    Recall that $(\sigma_1\dots,\sigma_r)$ is the moment map of $(K_1,\dots,K_r)$, which is dual to $(\theta_1,\dots,\theta_r)$. Now we want to change to new variables $(\sigma_1,p_{nc}(\alpha_1),\dots,p_{nc}(\alpha_{l-1}))$, which is the moment map of $(K_1,T_1,\dots,T_{l-1})$ up to some coefficients. That is why we define $(\beta_0,\dots,\beta_{l-1})$ as the dual of $(K_1,T_1,\dots,T_{l-1})$ (up to some coefficients).
\end{remark}

\subsection{Change of variable for $g_{\xi}^\p$}
The main object of this subsection is to prove the following identity:
\begin{proposition}
    We have
    \begin{align}\label{proposition g prime xi}
        g_{\xi}^\p = (d\sigma_1)^2 + \sum_{i=1}^{l-1}\frac{-1}{\Delta_i(\alpha_1,\dots,\alpha_{l-1})p_{nc}(\alpha_i)}(d(p_{nc}(\alpha_i)))^2.
    \end{align}
\end{proposition}
\begin{proof}
    First, we make a change of variable from $\xi_1,\dots,\xi_l$ to $\sigma_1,\dots,\sigma_l$.

    Note that
    \begin{align}
        d\xi_j = \frac{1}{\Delta(\xi_j)}\sum_{r=1}^l(-1)^r\xi_j^{l-r}d\sigma_r.
    \end{align}
    Applying the above formula to $g^\p_{\xi}$, we have
    \begin{align}
        g^\p_{\xi} = \sum_{r,s = 1}^lG_{rs}d\sigma_rd\sigma_s,
    \end{align}
    where
    \begin{align}
        G_{rs} = (-1)^{r+s}\sum_{j=1}^l\frac{\xi_j^{2l-r-s}}{\Delta(\xi_j)\prod_{k=1}^{l-1}(\xi_j-\alpha_k)}.
    \end{align}
    For any $m\geq 1$ and $a_1,\dots, a_m\in \mathbb{R}$, we have the following.
    \begin{align}
        \frac{1}{\prod_{i=1}^ma_i} = (-1)^{m+1}\sum_{i=1}^m\frac{1}{\prod_{k=1,k\neq i}^m(a_i-a_k)}\frac{1}{a_i}.
    \end{align}
    Applying this to $m=l-1$, $a_i = \xi_j-\alpha_i$, we have
    \begin{align}
        \frac{1}{\prod_{k=1}^{l-1}(\xi_j-\alpha_i)} = \sum_{i=1}^{l-1}\frac{1}{\Delta_i(\alpha_1,\dots,\alpha_{l-1})(\xi_j-\alpha_i)}.
    \end{align}
    Thus, we get
    \begin{align}
        G_{rs} &= (-1)^{r+s}\sum_{i=1}^{l-1}\frac{1}{\Delta_i(\alpha_1,\dots,\alpha_{l-1})}\sum_{j=1}^l\frac{\xi_j^{l-r-s}}{\Delta(\xi_j)(\xi_j-\alpha_i)}\\
        &= (-1)^{r+s}\sum_{i=1}^{l-1}\frac{1}{\Delta_i(\alpha_1,\dots,\alpha_{l-1})}\left[\sum_{k=0}^{l-r-s}h_{l-r-s-k}\alpha_i^{k} - \frac{\alpha_i^{2l-r-s}}{p_{nc}(\alpha_i)} \right].
    \end{align}
    In the last step, we have used \eqref{equation vandermonde identity complicated extended} and the summation $\sum_{k=0}^{l-r-s}h_{l-r-s-k}\alpha_i^{k}$ is understood as $0$ if $l-r-s<0$. Now applying \eqref{equation vandermonde identity basic}, we obtain
    \begin{align}
        G_{11} = 1 + \sum_{i=1}^{l-1}\frac{-1}{\Delta_i(\alpha_1,\dots,\alpha_{l-1})}\frac{\alpha_i^{2l-2}}{p_{nc}(\alpha_i)},
    \end{align}
    and for $(r,s)\neq (1,1)$, we have
    \begin{align}
        G_{rs} = (-1)^{r+s}\sum_{i=1}^{l-1}\frac{-1}{\Delta_i(\alpha_1,\dots,\alpha_{l-1})}\frac{\alpha_i^{2l-r-s}}{p_{nc}(\alpha_i)}.
    \end{align}
    It follows that
    \begin{align}
        g^\p_{\xi} &= (d\sigma_1)^2 + \sum_{i=1}^{l-1}\frac{-1}{\Delta_i(\alpha_1,\dots,\alpha_{l-1})p_{nc}(\alpha_i)}\sum_{r,s=1}^l(-1)^{r+s}\alpha_i^{2l-r-s}d\sigma_rd\sigma_s\\
        &= (d\sigma_1)^2 + \sum_{i=1}^{l-1}\frac{-1}{\Delta_i(\alpha_1,\dots,\alpha_{l-1})p_{nc}(\alpha_i)}\left(\sum_{r=1}^l(-1)^r\alpha_i^{l-r}d\sigma_r \right)^2\\
        &= (d\sigma_1)^2 + \sum_{i=1}^{l-1}\frac{-1}{\Delta_i(\alpha_1,\dots,\alpha_{l-1})p_{nc}(\alpha_i)}(d(p_{nc}(\alpha_i)))^2.
    \end{align}
\end{proof}
\subsection{Locally flatness of $g^\p$}
Combining Proposition \ref{proposition g prime theta} and Proposition \ref{proposition g prime xi}, we conclude that
\begin{align*}
    g^\p =& \sum_{j=1}^{l-1}(-1)^{l-j}p_{nc}(\alpha_j)\check{g}_j + (d\sigma_1)^2 + \sum_{j=1}^{l-1}\frac{-1}{\Delta_j(\alpha_1,\dots,\alpha_{l-1})p_{nc}(\alpha_j)}(d(p_{nc}(\alpha_j)))^2 +\\
    &+\beta_0^2 + \sum_{j=1}^{l-1}\frac{-p_{nc}(\alpha_j)}{\Delta_j(\alpha_1,\dots,\alpha_{l-1})}\beta_j^2.
\end{align*}
Note that by Lemma \ref{lemma d beta_i}, we have $2\check{g}^0_j + \beta_j^2 = 4g_{\mathbb{S}_1^{2d_j+1}}$ for $j=1,\dots,l-1$, where $g_{\mathbb{S}_1^{2d_j+1}}$ stands for the standard round metric of the sphere of dimension $2d_j + 1$ and radius $1$. Hence we have
\begin{align*}
    g^\p =& (d\sigma_1)^2 + \beta_0^2 +\\
    &+\sum_{j=1}^{l-1}\left[\frac{-p_{nc}(\alpha_j)}{\Delta_j(\alpha_1,\dots,\alpha_{l-1})}(4g_{\mathbb{S}_1^{2d_j+1}}) + \frac{-1}{\Delta_j(\alpha_1,\dots,\alpha_{l-1})p_{nc}(\alpha_j)}(d(p_{nc}(\alpha_j)))^2 \right].
\end{align*}
For $j=1,\dots,l-1$, define
\begin{align}\label{equation definition of r_j}
    r_j = 2\sqrt{\frac{-p_{nc}(\alpha_j)}{\Delta_j(\alpha_1,\dots,\alpha_{l-1})}},
\end{align}
then
\begin{align}
    g^\p = (d\sigma_1)^2 + \beta_0^2 + \sum_{j=1}^{l-1}(r_j^2g_{\mathbb{S}_1^{2d_j+1}} + (dr_j)^2).
\end{align}
Since $\mathbb{S}_1^{2d_j+1}$ is the link of the flat cone $\mathbb{R}^{2d_j+2}$, we conclude that
\begin{align}\label{equation g prime is locally flat}
    g^\p = (d\sigma_1)^2 + \beta_0^2 + \sum_{j=1}^{l-1}g_{\mathbb{R}^{2d_j+2}}.
\end{align}
Recall that by Lemma \ref{lemma d beta_0}, the $1$-form $\beta_0$ is closed, so the above formula shows that $g^\p$ is indeed a locally flat metric.

As a by-product of the above proof, we have the following proposition.
\begin{proposition}\label{proposition length of T_j}
    There exists a positive constant $C>0$ depending only on $a,\alpha_1,\dots,\alpha_l$ such that for any $1\leq j \leq l-1$, we have
    \begin{align}
        g^\p(T_j,T_j) & = r_j^2,\\
        g^\p(T_l,T_l) &\leq C\left(1 + \sum_{j=1}^{l-1}r_j^2\right).
    \end{align}
\end{proposition}
For simplicity, in this article we will use $C$ to denote a positive constant which may be different from lines.

\section{The asymptotic cone of the locally flat metric}\label{section The asymptotic cone of the locally flat metric}
The locally flat metric $g^\p$ defined on $M^0$ is not complete, so it is a little subtle to talk about its asymptotic cone. Instead, we first show that $(M^0,g^\p)$ can be identified as an open subset of a complete locally flat metric, and then determine the asymptotic cone of the complete metric.

Recalling equation \eqref{equation expand e_1 in v_j}, we have
\begin{align}\label{equation first recurrence}
    \frac{2\prod_{k=1}^{l-1}(\alpha_l-\alpha_k)}{q(\alpha_l)} e_1 =
    -v_l + \sum_{j=1}^{l-1}(-1)^{l-j}\frac{\prod_{k=1,k\neq j}^{l-1}(\alpha_l - \alpha_k)}{q(\alpha_l)}v_j.
\end{align}
Note that the coefficient of $e_1$ is strictly positive. For $j=1,\dots,l-1$, let $\tau_j$ be the coefficient of $v_j$ in the above formula,
\begin{align}
    \tau_j = (-1)^{l-j}\frac{\prod_{k=1,k\neq j}^{l-1}(\alpha_l - \alpha_k)}{q(\alpha_l)},
\end{align}
and define $\tau = (\tau_1,\dots,\tau_{l-1})\in \mathbb{R}^{l-1}$. Let $\mathbb{Z}^{l-1}\subset \mathbb{R}^{l-1}$ be the standard integer lattice in $\mathbb{R}^{l-1}$, then we can also view $\tau$ as an element in the $(l-1)$-dimensional torus $\mathbb{T}^{l-1} = \mathbb{R}^{l-1}/\mathbb{Z}^{l-1}$.

If we change the parameters from $(a, \alpha_1,\dots,\alpha_l)$ to $(\frac{a}{c},c\alpha_1+d,\dots, c\alpha_l+d)$, then we can verify that $\tau$ remains unchanged. Thus, $\tau$ is intrinsically associated with the isometry class of the steady gradient K\"ahler-Ricci soliton of the Taub-NUT type.

The geometric meaning of $\tau$ can be explained by the Poincar\'e recurrence. Consider the continuous flow generated by $e_1$ in the $l$-dimensional torus $\mathbb{T}^l = \mathbb{R}^l/\Gamma_v$. By \eqref{equation first recurrence}, the flow is transverse to the $(l-1)$-dimensional subtorus $\mathbb{T}^{l-1}_{v_1,\dots,v_{l-1}}$ spanned by $v_1,\dots,v_{l-1}$. Choose this subtorus as the Poincar\'e section and start the flow at $0$, then the first recurrence is $\sum_{j=1}^{l-1}\tau_jv_j$. Identify this subtorus with $\mathbb{T}^{l-1} = \mathbb{R}^{l-1}/\mathbb{Z}^{l-1}$, then the first recurrence map is the translation by $\tau$.

Let $\Lambda\subset \mathbb{T}^{l-1}$ be the closure of the subgroup generated by $\tau$, then we have
\begin{proposition}\label{proposition dimension of Lambda}
    The dimension of $\Lambda$ is given by
    \begin{align}
        \dim\Lambda &= \dim_\mathbb{Q}\op{Span}_\mathbb{Q}\{1,\tau_1,\dots,\tau_{l-1}\} - 1\\
        &= \dim_\mathbb{Q}\op{Span}_\mathbb{Q}\{(-1)^{l+1-j}\frac{q(\alpha_j)}{2(d_j + 1)}\frac{1}{\prod_{k=1,k\neq j}^l(\alpha_j-\alpha_k)}|j=1,\dots,l\} -1
    \end{align}
\end{proposition}
\begin{proof}
    It is well known that the orbit of $\tau$ is dense if and only if $1,\tau_1,\dots,\tau_{l-1}$ are $\mathbb{Q}$-independent. More generally, the result can be proved by this special case and Gauss elimination.
\end{proof}

Now we can state the first main result of this section.
\begin{proposition}\label{proposition isometric embedding of g prime}
    The Riemannian manifold $(M^0,g^\p)$ can be isometrically embedded into an open subset of $((\prod_{j=1}^{l-1}\mathbb{C}^{d_j+1}\times \mathbb{R})/\mathbb{Z})\times \mathbb{R}$ equipped with the standard Euclidean metric. Here, the $\mathbb{Z}$-action on $(\prod_{j=1}^{l-1}\mathbb{C}^{d_j+1}\times \mathbb{R})$ is defined as follows: On each factor $\mathbb{C}^{d_j+1}$ there is an action of $\mathbb{S}^1$ by rotation, so we have an $\mathbb{T}^{l-1} = \mathbb{R}^{l-1}/\mathbb{Z}^{l-1}$-action on $\prod_{j=1}^{l-1}\mathbb{C}^{d_j+1}$. Define the generator of the $\mathbb{Z}$-action as $\tau\in\mathbb{T}^{l-1}$ on $\prod_{j=1}^{l-1}\mathbb{C}^{d_j+1}$ and the translation by $ (-1)\frac{2}{q(\alpha_l)}{\prod_{k=1}^{l-1}(\alpha_l-\alpha_k)}$ on the $\mathbb{R}$ factor, then clearly this action is free and properly discontinuous. The image of the embedding of $(M^0,g^\p)$ is given by the following inequalities: $r_j>0$ for $j=1,\dots,l-1$ and $\frac{1}{4}\sum_{j=1}^{l-1}\frac{1}{\alpha_l-\alpha_j}r_j^2 > -\sigma$, where $r_j$ is the radius of the cone $\mathbb{C}^{d_j+1}$ and $\sigma$ is the coordinate of the last factor of $\mathbb{R}$.
\end{proposition}
\begin{proof}
    On $\prod_{j=1}^{l-1}\mathbb{C}^{d_j+1}\times \mathbb{R} \times \mathbb{R}$ equipped with the standard Euclidean metric, denote by $T_j$ the generator of rotation on $\mathbb{C}^{d_j+1}$ for $j=1,\dots,l-1$ and $K_1$ the unit vector field in the first component of $\mathbb{R}$. Then we can define a new vector field $T_l$ using formula \eqref{equation K_1 in terms of T_j}. Consequently, we have
    \begin{align}
        T_l = (-1)\frac{2}{q(\alpha_l)}{\prod_{k=1}^{l-1}(\alpha_l-\alpha_k)}K_1 + \sum_{j=1}^{l-1}\tau_jv_j.
    \end{align}

    Let the $1$-form $\beta_0$ be the metric dual to $K_1$, then $\beta_0(K_1)=1$ and $\beta_0$ vanishes when restricted to $\prod_{j=1}^{l-1}\mathbb{C}^{d_j+1}$ and the last component of $\mathbb{R}$, moreover $d\beta_0 = 0$. It follows that formula \eqref{equation g prime is locally flat} defines the standard Euclidean metric on $\prod_{j=1}^{l-1}\mathbb{C}^{d_j+1}\times \mathbb{R} \times \mathbb{R}$, where we impose $\sigma = \sigma_1 - \sum_{k=1}^l\alpha_k$.

    Recall that $\mathbb{T}^{l-1}_{v_1,\dots,v_{l-1}}$ is the subtorus in $\mathbb{T}^l$ spanned by $v_1,\dots,v_{l-1}$ and can be identified with $\mathbb{T}^{l-1} = \mathbb{R}^{l-1}/\mathbb{Z}^{l-1}$ by sending $\sum_{i=1}^{l-1}\mu_iv_i$ to $\mu = (\mu_1,\dots,\mu_{l-1})$. So we have a generically free $\mathbb{T}^{l-1}_{v_1,\dots,v_{l-1}}$-action on $\prod_{j=1}^{l-1}\mathbb{C}^{d_j+1}\times \mathbb{R}$.

    However, we cannot embed $M^0$ into $\prod_{j=1}^{l-1}\mathbb{C}^{d_j+1}\times \mathbb{R} \times \mathbb{R}$ since there is no generically free $\mathbb{T}^l$-action on $\prod_{j=1}^{l-1}\mathbb{C}^{d_j+1}\times \mathbb{R} \times \mathbb{R}$. To extend the $\mathbb{T}^{l-1}_{v_1,\dots,v_{l-1}}$-action to a generically free $\mathbb{T}^l$-action, we need to require that $T_l$ generates a generically free $\mathbb{S}^1$-action. Thus, we have to take a further quotient by $\mathbb{Z}$ described in the statement of the proposition.

    Now we can view $(M^0,g^\p)$ as a subset of $((\prod_{j=1}^{l-1}\mathbb{C}^{d_j+1}\times \mathbb{R})/\mathbb{Z})\times \mathbb{R}$ and it remains to precisely determine this subset. Recall that in the definition of $M^0$, we have $(\xi_1,\dots,\xi_l)\in \mathring{D} = (-\infty,\alpha_1)\times(\alpha_1,\alpha_2)\times\dots\times(\alpha_{l-2},\alpha_{l-1})\times(\alpha_l,+\infty)$, which is equivalent to $(-1)^{l-j}p_{nc(\alpha_j)} >0$ for $j=1,\dots,l-1$ and $p_{nc}(\alpha_l) <0$. By \eqref{equation definition of r_j}, the first $l-1$ inequalities are equivalent to $r_j>0$ for $j=1,\dots,{l-1}$. For the last inequality, by equations \eqref{equation beginning of p_nc alpha_l in terms of p_nc alpha_j}-\eqref{equation ending of p_nc alpha_l in terms of p_nc alpha_j}, we have
    \begin{align}
        \frac{p_{nc}(\alpha_l)}{\Delta_l(\alpha_1,\dots,\alpha_l)} =
        \sum_{k=1}^l\alpha_k - \sigma_1 - \sum_{j=1}^{l-1}\frac{p_{nc}(\alpha_j)}{\Delta_j(\alpha_1,\dots,\alpha_l)}.
    \end{align}
    Thus, $p_{nc}(\alpha_l)<0$ is equivalent to
    \begin{align}
        \frac{1}{4}\sum_{j=1}^{l-1}\frac{1}{\alpha_l-\alpha_j}r_j^2 > -\sigma,
    \end{align}
    where $\sigma = \sigma_1 - \sum_{k=1}^l\alpha_k$.
\end{proof}
\begin{remark}
    The Riemannian manifold $((\prod_{j=1}^{l-1}\mathbb{C}^{d_j+1}\times \mathbb{R})/\mathbb{Z})\times \mathbb{R}$ is complete, according to the Hopf-Rinow theorem.
\end{remark}
\begin{remark}
    If we view $(M^0,g^\p)$ as a subset of $((\prod_{j=1}^{l-1}\mathbb{C}^{d_j+1}\times \mathbb{R})/\mathbb{Z})\times \mathbb{R}$, then for $j=1,\dots,l-1$, $\beta_j$ is the metric dual of $T_j$.
\end{remark}

Next, we determine the asymptotic cone of $((\prod_{j=1}^{l-1}\mathbb{C}^{d_j+1}\times \mathbb{R})/\mathbb{Z})\times \mathbb{R}$.
\begin{proposition}\label{proposition asymptotic cone of g prime}
    The asymptotic cone of $((\prod_{j=1}^{l-1}\mathbb{C}^{d_j+1}\times \mathbb{R})/\mathbb{Z})\times \mathbb{R}$ is unique and is $(\prod_{j=1}^{l-1}\mathbb{C}^{d_j+1}/\Lambda)\times \mathbb{R}$, where $\Lambda$ acts on $\prod_{j=1}^{l-1}\mathbb{C}^{d_j+1}$ as a subgroup of $\mathbb{T}^{l-1}$.
\end{proposition}
\begin{proof}
    Consider the following map.
    \begin{align}
        f: ((\prod_{j=1}^{l-1}\mathbb{C}^{d_j+1}\times \mathbb{R})/\mathbb{Z})\times \mathbb{R} \rightarrow
        (\prod_{j=1}^{l-1}\mathbb{C}^{d_j+1}/\Lambda)\times \mathbb{R},
    \end{align}
    which is the quotient map of the equivalence relationship whose equivalence class are the closures of the $\mathbb{R}$-action generated by $K_1$. Equivalently, let $\Lambda^+$ be the subtorus in $\mathbb{T}^l$ spanned by $\Lambda$ and $e_1$, then it is a closed subtorus of dimension $\dim\Lambda +1$, and $f$ is the quotient map with respect to the action of $\Lambda^+$. Since $K_1$ is a Killing vector field, we can equip the target $(\prod_{j=1}^{l-1}\mathbb{C}^{d_j+1}/\Lambda)\times \mathbb{R}$ with the quotient metric and consequently $f$ is a submetry. In any open subset where $\Lambda^+$ acts freely, $f$ is simply a Riemannian submersion.

    Note that $\Lambda$ acts on $\prod_{j=1}^{l-1}\mathbb{C}^{d_j+1}$ in a way that preserves the link, so the target $(\prod_{j=1}^{l-1}\mathbb{C}^{d_j+1}/\Lambda)\times \mathbb{R}$ is a metric cone. For any $\lambda>0$, let $f_\lambda = \lambda f$. In the following, we will denote by $g_{Euc}$ the Euclidean metric on $((\prod_{j=1}^{l-1}\mathbb{C}^{d_j+1}\times \mathbb{R})/\mathbb{Z})\times \mathbb{R}$, then
    \begin{align}
        f_\lambda : (((\prod_{j=1}^{l-1}\mathbb{C}^{d_j+1}\times \mathbb{R})/\mathbb{Z})\times \mathbb{R}, \lambda^2g_{Euc}) \rightarrow (\prod_{j=1}^{l-1}\mathbb{C}^{d_j+1}/\Lambda)\times \mathbb{R}
    \end{align}
    is also a submetry.

    We claim that for any small $\epsilon > 0$, there exists $\lambda_0 > 0$ such that for any $0<\lambda <\lambda_0$, $f_\lambda$ is a pointed $\epsilon$ Gromov-Hausdorff approximation ($\epsilon$-GHA) of $(((\prod_{j=1}^{l-1}\mathbb{C}^{d_j+1}\times \mathbb{R})/\mathbb{Z})\times \mathbb{R}, \lambda^2g_{Euc})$. Here $\epsilon-$GHA means that we need to show the following two properties:
    \begin{itemize}
        \item ($\epsilon$-onto) $B_{\epsilon^{-1}}(0) \subset B_{\epsilon}(f_\lambda(B_{\epsilon^{-1}}(0)))$,
        \item ($\epsilon$-isometry) For any $x_1,x_2 \in B_{\epsilon^{-1}}(0) \subset (((\prod_{j=1}^{l-1}\mathbb{C}^{d_j+1}\times \mathbb{R})/\mathbb{Z})\times \mathbb{R}, \lambda^2g_{Euc})$, we have
    \begin{align}
        |d_{\lambda^2g_{Euc}}(x_1,x_2) - d(f_\lambda(x_1), f_\lambda(x_2) )| < \epsilon,
    \end{align}
    which is equivalent to
    \begin{align}
        \lambda|d_{g_{Euc}}(x_1,x_2) - d(f(x_1), f(x_2) )| < \epsilon.
    \end{align}
    \end{itemize}

    Since $f_\lambda$ is a submetry, we have $B_{\epsilon^{-1}}(0) = f_\lambda(B_{\epsilon^{-1}}(0))$. In particular, it is $\epsilon$-onto.

    Since $f_\lambda$ is a submetry, minimal geodesic in the base can be lifted. So it remains to show the following statement: There exists $C>0$ such that for any $\epsilon >0$, there exists $\lambda_0 >0$ such that for any $m \in (\prod_{j=1}^{l-1}\mathbb{C}^{d_j+1}/\Lambda)\times \mathbb{R}$ with $|m| = d(0,m) < \epsilon^{-1}$ and $0<\lambda<\lambda_0$, we have $\op{diam}_{\lambda^2g_{Euc}}(f_\lambda^{-1}(m)) < C\epsilon$.

    Equivalently, it suffices to show that there exists $C>0$ such that for any $\epsilon >0$, there exists $\lambda_0 >0$ such that for any $m \in (\prod_{j=1}^{l-1}\mathbb{C}^{d_j+1}/\Lambda)\times \mathbb{R}$ with $|m|< \lambda^{-1}\epsilon^{-1}$ and $0<\lambda<\lambda_0$, we have $\op{diam}_{g_{Euc}}(f^{-1}(m)) < C\lambda^{-1}\epsilon$.

    Let $x\in ((\prod_{j=1}^{l-1}\mathbb{C}^{d_j+1}\times \mathbb{R})/\mathbb{Z})\times \mathbb{R}$ such that $f(x) = m$, then $f^{-1}(m)$ is the $\Lambda^+$-orbit $\Lambda^+\cdot x$ of $m$. Since $K_1$ is of unit length, it suffices to show that $\op{diam}_{g_{Euc}}(\Lambda\cdot x) < C\lambda^{-1}\epsilon$. We will estimate $d_{g_{Euc}}(x,t\cdot x)$ for any $t\in \Lambda$.

    Since $\mathbb{Z}\tau$ is dense in $\Lambda$ and $\Lambda$ is compact, there exists a positive integer $N$ such that $d_{\mathbb{T}^{l-1}}(\{s\tau|s\in \mathbb{Z},|s|\leq N\}, t) < \epsilon^2$ for any $t \in \Lambda$. Let $\lambda_0 = \frac{\epsilon}{N}$, note that it depends only on $\epsilon$ and $\tau$. Then for any $t\in \Lambda$ and $0<\lambda <\lambda_0$, there exists $n_0\in \mathbb{Z}$ such that $|n_0|\leq N < \lambda^{-1}\epsilon$ and $d_{\mathbb{T}^{l-1}}(t,n_0\tau) < \epsilon^2$.

    Note that $|m|^2 = \sigma(m)^2 + \sum_{j=1}^{l-1}r_j(m)^2 = \sigma(x)^2 + \sum_{j=1}^{l-1}r_j(x)^2$, so by Proposition \ref{proposition length of T_j}, for $j=1,\dots,l-1$, we have $|T_j|_{g_{Euc}}\leq |m| < \lambda^{-1}\epsilon^{-1}$.

    Now we join $x$ and $t\cdot x$ by curves in two steps. First we join $x$ and $(n_0\tau)\cdot x$ by the flow along $K_1$, it has length $|n_0|<\lambda^{-1}\epsilon$. Next we join $(n_0\tau)\cdot x$ and $t\cdot x$ by the flow of $T_j$ for $j=1,\dots,l-1$, the length of the curve is bounded by $|m|d_{\mathbb{T}^{l-1}}(t,n_0\tau) < (l-1)\lambda^{-1}\epsilon^{-1}\epsilon^2 = (l-1)\lambda^{-1}\epsilon$. In conclusion, we have $\op{diam}_{g_{Euc}}(\Lambda\cdot x) < l\lambda^{-1}\epsilon$, which proves the claim.

    By the claim, we deduce that $(\prod_{j=1}^{l-1}\mathbb{C}^{d_j+1}/\Lambda)\times \mathbb{R}$ is an asymptotic cone of $((\prod_{j=1}^{l-1}\mathbb{C}^{d_j+1}\times \mathbb{R})/\mathbb{Z})\times \mathbb{R}$. Moreover, since the Gromov-Hausdorff distance is complete on the collection of all isometric classes of proper and complete pointed metric spaces, our claim implies that the asymptotic cone is unique.
\end{proof}
\begin{remark}
    If we view $M^0$ as an open subset of $((\prod_{j=1}^{l-1}\mathbb{C}^{d_j+1}\times \mathbb{R})/\mathbb{Z})\times \mathbb{R}$, then its image $f_\lambda(M^0)$ under $f_\lambda$ is given by inequalities $r_j>0$ for $j=1,\dots,l-1$ and $\frac{1}{4}\sum_{j=1}^{l-1}\frac{1}{\alpha_l-\alpha_j}r_j^2>-\lambda\sigma$. Thus, $f_\lambda|_{M^0}$ is also $\epsilon$-onto if we let $\lambda$ sufficiently small, hence an $\epsilon$-GHA. So roughly speaking, $(\prod_{j=1}^{l-1}\mathbb{C}^{d_j+1}/\Lambda)\times \mathbb{R}$ is an asymptotic cone of $(M^0,g^\p)$.
\end{remark}
For simplicity, denote by $(E, g_{Euc})$ the Riemannian manifold $((\prod_{j=1}^{l-1}\mathbb{C}^{d_j+1}\times \mathbb{R})/\mathbb{Z})\times \mathbb{R}$, and by $E/\Lambda^+$ its asymptotic cone $(\prod_{j=1}^{l-1}\mathbb{C}^{d_j+1}/\Lambda)\times \mathbb{R}$. So $f:E\rightarrow E/\Lambda^+$ is the quotient map and also a submetry. It is clear that the radius of $E/\Lambda^+$ is given by $\sqrt{\sum_{j=1}^{l-1}r_j^2 + \sigma^2}$. And by the proof of the previous proposition, we have the following.
\begin{proposition}\label{proposition rho prime is proportional to cone radius}
    Denote by $\rho^\prime(x)$ the distance function $d_{g_{Euc}}(0,x)$ of $E$, then outside a compact set, we have
    \begin{align}
        \sqrt{\sum_{j=1}^{l-1}r_j^2 + \sigma^2} \leq \rho^\p \leq C\sqrt{\sum_{j=1}^{l-1}r_j^2 + \sigma^2},
    \end{align}
    for some constant $C>0$.
\end{proposition}
\begin{proof}
    Since $f$ is a submetry, we have $\sqrt{\sum_{j=1}^{l-1}r_j^2 + \sigma^2} \leq \rho^\p$. And by the proof of Proposition \ref{proposition asymptotic cone of g prime}, we know that the $\Lambda^+$-orbit of $x$ has a diameter much smaller than $\rho^\p(x)$, proving the other inequality.
\end{proof}

\section{The asymptotic cone of the K\"ahler-Ricci soliton}\label{section The asymptotic cone of the Kahler-Ricci soliton}
In this section we show that the asymptotic cone of $(\mathbb{C}^n,g)$ is $E/\Lambda^+$. To do this, we first need an estimation of the distance function of $g$. Note that it is known that $\xi_1,\dots,\xi_l$ extend to smooth functions on $\mathbb{C}^n$.
\begin{proposition}[{\cite[Lemma 5.8]{apostolov2023hamiltonian}}]\label{proposition rho is comparable to xi_l-xi_1}
    Denote by $\rho(x) = d_g(0,x)$ the distance function on $(\mathbb{C}^n,g)$. Then outside a compact set, we have
    \begin{align}
        \frac{1}{C}(\xi_l-\xi_1) \leq \rho \leq C(\xi_l-\xi_1),
    \end{align}
    for some constant $C>0$.
\end{proposition}
For simplicity, we will express this proposition by $\rho \sim (\xi_l-\xi_1)$. By this we mean that $\frac{\rho}{\xi_l-\xi_1}$ is bounded from above and below by positive constants outside a compact set. Or equivalently, it means that $\frac{1 + \rho}{1 + (\xi_l - \xi_1)}$ is bounded from above and below by positive constants everywhere. So, Proposition \ref{proposition rho prime is proportional to cone radius} implies that $\rho^\p \sim \sqrt{\sum_{j=1}^{l-1}r_j^2 + \sigma^2}$.

Since we want to prove that $g$ and $g^\p$ have the same asmptotic cone, we should compare them.
\begin{proposition}\label{proposition g-g prime bounded by xi_l}
    We have
    \begin{align}
        |g-g^\p|_g = O(\frac{1}{\xi_l^{n-1}})
    \end{align}
    as $\xi_l \rightarrow +\infty$.
\end{proposition}
\begin{proof}
    Comparing the difference between $g^\p$ and $g$, we find that $g^\p$ is obtained from $g$ by replacing $F_l(\xi_l)$ with $P(\xi_l)$. Now the result follows from
    \begin{align}
        \frac{F_l({\xi_l})}{P(\xi_l)} - 1 = e^{2a(\alpha_l-\xi_l)}\frac{P(\alpha_l)}{P(\xi_l)},
    \end{align}
    and the fact that $a\geq 0$ and $P(t)$ is a polynomial of degree $n-1$.
\end{proof}

For $0<\alpha<1$ and $c>0$, define $R_{\alpha,c} = \{x\in \mathbb{C}^n|\xi_l(x) > c(\xi_l(x) - \xi_1(x))^\alpha\}$, and $S_{\alpha,c}$ the complement of $R_{\alpha,c}$ in $\mathbb{C}^n$. Consequently, we have
\begin{proposition}\label{proposition g-g prime in regular region}
    In $R_{\alpha,c}$, we have
    \begin{align}
        |g-g^\p|_g = O(\frac{1}{\rho^{\alpha(n-1)}})
    \end{align}
    as $\rho \rightarrow +\infty$.
\end{proposition}
\begin{proof}
    It is a consequence of Proposition \ref{proposition rho is comparable to xi_l-xi_1}, Proposition \ref{proposition g-g prime bounded by xi_l} and the definition of $R_{\alpha,c}$.
\end{proof}
Recall the definitions of $r_j^2$ and $\sigma$, since $\xi_1,\dots,\xi_l$ can be extended to smooth functions on $\mathbb{C}^n$, as are $r_j^2$ and $\sigma$.
\begin{proposition}\label{proposition rho is comparable to cone radius}
    As functions on $\mathbb{C}^n$, we have $\rho \sim \sqrt{\sum_{j=1}^{l-1}r_j^2 + \sigma^2}$.
\end{proposition}
\begin{proof}
    By the definitions of $r_j^2$ and $\sigma$, we have
    \begin{align}
        |\sigma| &= |\sum_{j=1}^l\xi_j - \sum_{j=1}^l \alpha_l| \leq \xi_l - \xi_1 + C,\\
        r_j^2 &= C|p_{nc}(\alpha_j)| = C(\alpha_j-\xi_1)(\xi_l-\alpha_j) \leq C(\xi_l-\xi_1)^2.
    \end{align}
    Thus, $\sqrt{\sum_{j=1}^{l-1}r_j^2 + \sigma^2} \leq  C(1 + \rho)$.

    Conversely, for any $x\in \mathbb{C}^n$, we will find a piecewise smooth curve in $\mathbb{C}^n$ that joins $x$ and the origin. Recall that $\sigma_1$ is the moment map of the action generated by $K_1$, so we can use the flow of $JK_1$ to join $x$ with another point $x^\p$ such that $\sigma(x^\p) = 0$. By construction, the length $l_g(\gamma_1)$ measured by $g$ of this curve $\gamma_1$ is $|\sigma(x)|$. Recall that $g^\p(K_1,K_1) = 1$, comparing $g^\p$ and $g$, we have
    \begin{align}
        g(K_1,K_1) = 1 - \frac{P(\alpha_l)}{p_c(\xi_l)\Delta(\xi_l)} \geq
        1 - \frac{C}{\xi_l-\xi_1}.
    \end{align}
    It follows that there exists $\rho_0>0$ such that on $\{\rho > \rho_0\}$, we have $g(K_1,K_1) \geq \frac{1}{2}$.

    If $\gamma_1$ intersects with $\{\rho \leq \rho_0\}$, then clearly we have
    \begin{align}
    \rho(x) \leq |\sigma(x)| + \rho_0 \leq C\left(1 + \sqrt{\sum_{j=1}^{l-1}r_j^2(x) + \sigma^2(x)}\right).
    \end{align}

    If $\gamma_1$ is contained in $\{\rho > \rho_0\}$, then we know that $l_{g^\p}(\gamma_1)\leq 2|\sigma(x)|$. In this case, we consider the following.

    Inside $\{\sigma=0\}$, we have $-\xi_1 \sim \xi_l \sim \rho$, so we have $|g - g^\p|_g = O(\frac{1}{\rho^{(n-1)}})$ as $\rho\rightarrow +\infty$. Let $\gamma_2$ be the minimal geodesic with respect to $g^\p$ in $E$ that joins $x^\p$ and $0$, since $\{\sigma=0\}$ is a totally geodesic submanifold of $E$, $\gamma_2$ is contained in $\{\sigma=0\}$. We know that the length $l_{g^\p}(\gamma_2)$ of $\gamma_2$ measured by $g^\p$ is $\rho^\p(x^\p)$, so we have
    \begin{align}
        l_{g^\p}(\gamma_2) = \rho^\p(x^\p) \leq \rho^\p(x) + 2|\sigma(x)|\leq
        C\left(1 + \sqrt{\sum_{j=1}^{l-1}r_j^2(x) + \sigma(x)^2}\right).
    \end{align}
    In the last inequality, we have used Proposition \ref{proposition rho prime is proportional to cone radius}.

    Since $|g - g^\p|_g = O(\frac{1}{\rho^{(n-1)}})$ as $\rho\rightarrow +\infty$ inside $\{\sigma=0\}$, there exists $\rho_1 >0$ such that in $\{\rho\geq \rho_1\}\cap\{\sigma=0\}$, we have $g \leq 2g^\p$.

    Write $\gamma_2:[0,T] \rightarrow \mathbb{C}^n$, then $T =l_{g^\p}(\gamma_2) = \rho^\p(x^\p)$. Let $t_0\in [0,T]$ so that for any $0\leq t<t_0$, we have $\rho(\gamma_2(t)) > \rho_1$ and $\rho(\gamma_2(t_0))=\rho_1$. Then
    \begin{align}
        \rho(x) &\leq l_g(\gamma_1) + l_g(\gamma_2|_{[0,t_0]}) + \rho_1\\
        &\leq l_g(\gamma_1) + 2l_{g^\p}(\gamma_2) + \rho_1\\
        &\leq C\left(1 + \sqrt{\sum_{j=1}^{l-1}r_j^2(x) + \sigma(x)^2}\right).
    \end{align}

    Thus, we have shown that in any case we always have $\sqrt{\sum_{j=1}^{l-1}r_j^2 + \sigma^2} \leq  C(1 + \rho)$ and $\rho \leq C\left(1 + \sqrt{\sum_{j=1}^{l-1}r_j^2 + \sigma^2}\right)$, which finish the proof.
\end{proof}

Observe that we have two different metrics $g$ and $g^\p$ on $M^0$. When equipped with $g$, we know that $M^0$ is dense in $\mathbb{C}^n$. When equipped with $g^\p$, we know that $M^0$ can be isometrically embedded into an open subset of $E$. However, we cannot extend this embedding to a smooth map from $\mathbb{C}^n$ to $E$. Instead, we will fix an extension of this embedding that is not continuous as follows.

For any $x\in \mathbb{C}^n$, if $x\in M^0$, we define $\iota(x)$ as the image of $x$ under the embedding of $M^0$ into $E$. For any $x\in \mathbb{C}^n\setminus M^0$, fix a sequence $x_j$ in $M^0$ that converges to $x$, and define $\iota(x) = \lim_{j\rightarrow +\infty}\iota(x_j)$. Here, since $\rho(x_j)$ is bounded, by Proposition \ref{proposition rho is comparable to cone radius} and Proposition \ref{proposition rho prime is proportional to cone radius}, we know that $\rho^\p(\iota(x_j))$ is also bounded, so by passing to a subsequence, we may assume that $\iota(x_j)$ converges.

Moreover, we require $\iota(0) = 0$. The idea is that any point $x$ with $\iota(x) =0$ must have $\xi_j(x) = \alpha_j$ for $j=1,\dots,l$ by the definition of $r_j^2$ and $\sigma$. From the toric point of view, it means that $x$ is fixed by the action of $\mathbb{T}^l$, so it must be the origin of $\mathbb{C}^n$.

In this way, we get a map $\iota:\mathbb{C}^n \rightarrow E$ that extends the embedding of $M^0$ into $E$.
\begin{proposition}\label{proposition rho is comparable to rho prime}
    As functions on $\mathbb{C}^n$, we have $\rho\sim\iota^*\rho^\p$.
\end{proposition}
\begin{proof}
    It is a simple consequence of Proposition \ref{proposition rho is comparable to cone radius} and Proposition \ref{proposition rho prime is proportional to cone radius}.
\end{proof}
For simplicity, we will omit $\iota$ and simply write $\rho\sim \rho^\prime$. Intuitively, we think of $\mathbb{C}^n$ as a subset of $E$ by using $\iota$, although $\iota$ may not be injective.

By Proposition \ref{proposition g-g prime in regular region}, the metric $g$ is close to the locally flat metric $g^\p$ in the region $R_{\alpha,c}$. The following two lemmas discuss the property of the metric $g$ in the region $S_{\alpha,c}$.
\begin{lemma}\label{lemma connecting singular to regular}
Any point $x_0\in S_{\alpha,c}$ can be connected to $R_{\alpha,c}$ by a curve of length not greater than $C(\rho(x_0)^{\frac{1+\alpha}{2}} + 1)$, where $C>0$ is a constant depending on $c,\alpha$ and $\alpha_1,\dots,\alpha_l$. Similarly, it can also be connected to $\{\xi_l=\alpha_l\}$ by a curve of length not greater than $C(\rho(x_0)^{\frac{1+\alpha}{2}} + 1)$.
\end{lemma}
\begin{proof}
    Consider the curve $\gamma(t)$ in $\mathbb{C}^n$ passing through $x_0$ such that $\check{g}_j(\gamma^\p(t), -) = 0$, $\theta_r(\gamma^\p(t)) = 0$, $\xi_j(\gamma(t)) = \xi_j(x_0)$ and $\xi_l(\gamma(t)) = t$, where $j=1,\dots,l-1$ and $r=1,\dots,l$.

    Define $\xi_{k,0} = \xi_k(x_0)$ for $k=1,\dots,l$, then by the definition of $S_{\alpha,c}$, we have
    \begin{align}
        \xi_{l,0} \leq c(\xi_{l,0} - \xi_{1,0})^\alpha.
    \end{align}
    Choose $\xi_{l,1} > \xi_{l,0}$ such that
    \begin{align}
        c(\xi_{l,1} - \xi_{1,0})^\alpha < \xi_{l,1} < (c+1) (\xi_{l,1} - \xi_{1,0})^\alpha,
    \end{align}
    Then we will consider the curve $\gamma:[\alpha_l,\xi_{l,1}] \rightarrow \mathbb{C}^n$. Note that $\gamma(\alpha_l) \in \{\xi_l = \alpha_l\}$, $\gamma(\xi_{l,0}) = x_0$, and $\gamma(\xi_{l,1}) \in R_{\alpha,c}$. Let $L$ be the length of this curve; then we have
    \begin{align}
        L = \int_{\alpha_l}^{\xi_{l,1}}\sqrt{\frac{\prod_{j=1}^{l-1}(t-\alpha_j)^{d_j} \prod_{k=1}^{l-1}(t-\xi_{k,0})}{P(t) - e^{2a(\alpha_l-t)}P(\alpha_l)}}.
    \end{align}
    Write $P(t) - P(\alpha_l) = (t-\alpha_l)f(t)$, then $f(\alpha_l)>0$ since $P^\p(\alpha_t)>0$. In fact, we know that $f(t)$ is a polynomial of degree $n-2$ and $f(t) >0$ for any $t\geq \alpha_l$. So, there exists $C>0$ depending only on $\alpha_1,\dots,\alpha_l$ such that
    \begin{align}
        \frac{\prod_{j=1}^{l-1}(t-\alpha_j)^{d_j} \prod_{k=2}^{l-1}(t-\xi_{k,0})}{f(t)} \leq C,
    \end{align}
    for all $t\geq \alpha_l$. It follows that
    \begin{align}
        L &\leq C\int_{\alpha_l}^{\xi_{l,1}}\sqrt{\frac{t-\xi_{1,0}}{t-\alpha_l}}\\
        &= 2C\sqrt{\xi_{l,1} - \xi_{1,0}}\sqrt{\xi_{l,1} - \alpha_l}\\
        &\leq C(1 + \rho(x_1))^\frac{1+\alpha}{2},
    \end{align}
    where in the last line, $C$ depends on $c,\alpha_1,\dots,\alpha_l$.

    Define $\rho_0 = 1 + \rho(x_0)$ and $\rho_1 = 1 + \rho(x_1)$, then by the triangle inequality we have $|\rho_1-\rho_0|\leq L \leq C\rho_1^{\frac{1+\alpha}{2}}$.

    Assume first that $\rho_0 > R$ for some constant $R>0$, which will be determined later. We claim that for sufficiently large $R$, we have $\rho_1\leq 3\rho_0$. If it is not true, then $\rho_1>3\rho_0>3R$, and $|\rho_1-\rho_0|>\frac{1}{2}\rho_1$, so $\frac{1}{2}\rho_1\leq C\rho_1^{\frac{1+\alpha}{2}}$, and hence $\rho_1^\frac{1-\alpha}{2}\leq 2C$, implying that $(3R)^{\frac{1-\alpha}{2}}\leq 2C$. However, the last inequality is not true if $R$ is sufficiently large.

    So in this case, we have $L\leq C(1 + \rho(x_0))^{\frac{1+\alpha}{2}}$. Finally, the set $\{x_0\in\mathbb{C}^n|1 + \rho(x_0)\leq R\}$ is compact, so in any case there exists $C>0$ such that $L\leq C(1 + \rho(x_0))^{\frac{1+\alpha}{2}}$.

    Now, considering the segment of $\gamma$ corresponding to $t\in [\xi_{l,0}, \xi_{l,1}]$, we prove the statement about connecting $S_{\alpha,c}$ to $R_{\alpha,c}$. Similarly, considering the segment of $\gamma$ corresponding to $t\in [\alpha_l, \xi_{l,0}]$, we prove the statement about connecting $S_{\alpha,c}$ to $\{\xi_l=\alpha_l\}$.
\end{proof}
Intuitively, the above lemma says that $S_{\alpha_c}$ is thin. The next lemma says that it looks like a half-line.
\begin{lemma}\label{lemma singular region looks like a line}
    There exist $\rho_0>0$ and $C>0$ depending only on $c, \alpha,\alpha_1,\dots,\alpha_l$ such that for any points $x_1,x_2\in S_{\alpha,c}$ with $\rho(x_1),\rho(x_2)>\rho_0$ and $\xi_1(x_1)\leq \xi_1(x_2)$, we have $\xi_1(x_1),\xi_1(x_2) <0$ and
    \begin{align}
        0 \leq d_g(x_1,x_2) - (\xi_1(x_2) - \xi_1(x_1)) \leq C\left(\rho(x_1)^\frac{1+\alpha}{2} + \rho(x_2)^{\frac{1+\alpha}{2}} + \frac{\xi_1(x_2) - \xi_1(x_1)}{-\xi_1(x_2)}\right).
    \end{align}
\end{lemma}
\begin{proof}
    Recall that for $j=1,\dots,l-1$, we have $\xi_j\leq \alpha_j\leq \xi_{j+1}$, so in the formula \eqref{equation of g} that defines $g$, all the coefficients of $(d\xi_j)^2$ for $j=1,\dots,l-1$ are greater than $1$. In particular, considering the coefficient of $(d\xi_1)^2$, we get
    \begin{align}
        d_g(x_1,x_2) \geq \xi_1(x_2) - \xi_1(x_1),
    \end{align}
    which is the first inequality. To prove the second inequality, we note that using Lemma \ref{lemma connecting singular to regular}, we may assume that $x_1,x_2\in \{\xi_l=\alpha_l\}$.

    By equation \eqref{equation of g} again, letting $\xi_l = \alpha_l$, we get
    \begin{align}
        g|_{\{\xi_l=\alpha_l\}} =& \sum_{j=1}^{l-1}(-1)^{l-j} p_{nc}|_{\xi_l=\alpha_l}(\alpha_j)\check{g}_j + \sum_{j=1}^{l-1}\frac{(\xi_j-\alpha_l)\prod_{k=1,k\neq j}^{l-1}(\xi_j-\xi_k)}{\prod_{k=1}^{l-1}(\xi_j-\alpha_k)}(d\xi_j)^2 +\\
        &+\frac{\prod_{k=1}^{l-1}(\xi_j-\alpha_k)}{(\xi_j-\alpha_l)\prod_{k=1,k\neq j}^{l-1}(\xi_j-\xi_k)}\left(\sum_{r=1}^{l}\sigma_{r-1}|_{\xi_l=\alpha_l}(\hat{\xi}_j)\theta_r\right)^2
    \end{align}
    We will estimate the size (diameter) of each of its components.

    Firstly, by Proposition \ref{proposition rho is comparable to xi_l-xi_1} $(-1)^{l-j}p_{nc}|_{\xi_l=\alpha_l}(\alpha_j) \sim \rho$, so the size of the $\check{g}_j$ component is $O(\rho^{\frac{1}{2}})$ as $\rho \rightarrow +\infty$.

    Secondly, for $j=2,\dots,l-1$, the coefficient of $(d\xi_j)^2$ is bounded by
    \begin{align}
        C(\xi_j-\xi_1)\frac{1}{(\xi_j-\alpha_{j-1})(\alpha_j - \xi_j)}.
    \end{align}
    So, the size of the $(d\xi_j)^2$ component is $O(\rho^{\frac{1}{2}})$ as $\rho \rightarrow +\infty$.

    Thirdly, the coefficient of $\left(\sum_{r=1}^{l}\sigma_{r-1}|_{\xi_l=\alpha_l}(\hat{\xi}_1)\theta_r\right)^2$ is bounded by $1$ and we have $\sigma_{r-1}|_{\xi_l=\alpha_l}(\hat{\xi}_1) \sim 1$, so the size of the $\left(\sum_{r=1}^{l}\sigma_{r-1}|_{\xi_l=\alpha_l}(\hat{\xi}_1)\theta_r\right)^2$ component is $O(1)$ as $\rho \rightarrow +\infty$.

    Fourthly, for $i=2,\dots,l-1$, by \eqref{equation theta component of g prime}, when $\xi_l=\alpha_l$, we have
    \begin{align}
        \left|\left(\sum_{r=1}^{l}\sigma_{r-1}|_{\xi_l=\alpha_l}(\hat{\xi}_j)\theta_r\right)(T_j)\right| \leq C(1 + \rho)
    \end{align}
    Also, we note that the coefficient of $\left(\sum_{r=1}^{l}\sigma_{r-1}|_{\xi_l=\alpha_l}(\hat{\xi}_j)\theta_r\right)^2$ is of $O(\frac{1}{\rho})$, so the size of the $\left(\sum_{r=1}^{l}\sigma_{r-1}|_{\xi_l=\alpha_l}(\hat{\xi}_j)\theta_r\right)^2$ component is $O(\rho^{\frac{1}{2}})$ as $\rho \rightarrow +\infty$.

    Finally, consider a curve $\gamma: [\xi_1(x_1), \xi_2(x_1)] \rightarrow \{\xi_l = \alpha_l\}$ joining $x_1,x_2$, such that $\xi_1(\gamma(t)) = t$. To estimate the length of $\gamma$, it suffices to consider the contribution of the $(d\xi_1)^2$ component, since the contribution of the other components is at most $C(\rho(x_1)^\frac{1}{2} + \rho(x_2)^\frac{1}{2})$.

    Since $\rho \sim \xi_l-\xi_1$, when $\xi_l=\alpha$ we have $-\xi_1\sim \rho$ so by letting $\rho_0$ be large enough, we have $\xi_1(x_1),\xi_1(x_2)<0$. Examining the coefficient of $(d\xi_1)^2$, we find that it is bounded from above by $1 + \frac{C}{-\xi_1}$. Note that $\sqrt{1 + x}\leq 1 + \frac{1}{2}x$ for $x>0$, we get
    \begin{align}
        \int_{\xi_1(x_1)}^{\xi_1(x_2)}\sqrt{1 + \frac{C}{-t}}dt \leq
        \int_{\xi_1(x_1)}^{\xi_1(x_2)}(1 + \frac{C}{-2t})dt\leq
        (\xi_1(x_2) - \xi_1(x_1)) + C\frac{\xi_1(x_2) - \xi_1(x_1)}{-\xi_1(x_2)},
    \end{align}
    which proves the second inequality.
\end{proof}
Now it is time to prove the main result of this section.
\begin{theorem}\label{theorem asymptotic cone of g}
    The asymptotic cone of the K\"ahler-Ricci soliton $(\mathbb{C}^n,g)$ is unique and is $E/\Lambda^+ = (\prod_{j=1}^{l-1}\mathbb{C}^{d_j+1}/\Lambda)\times \mathbb{R}$.
\end{theorem}
\begin{proof}
    For any $\epsilon > 0$, we claim that there exists $\lambda_0>0$ such that for any $)<\lambda<\lambda_0$, the map $f_\lambda : \mathbb{C}^n \rightarrow E \rightarrow E/\Lambda^+$ is a pointed $\epsilon$-GHA of $\lambda^2g$. Here $f_\lambda$ is the composition of $\iota$, the submetry $f: E\rightarrow E/\Lambda^+$ and the rescaling by $\lambda$.

    Recall that $\iota(0) =0$, so $f_\lambda(0) = 0$ hence $f_\lambda$ is pointed.

    We will first show that $f_\lambda$ is $\epsilon$-isometry for any sufficiently small $\lambda$. More precisely, it means that for any two points $x_1,x_2\in \mathbb{C}^n$ such that $\rho(x_i)<(\lambda\epsilon)^{-1}$, we have
    \begin{align}
        |d_{\lambda^2g}(x_1,x_2) - d_{E/\Lambda^+}(f_\lambda(x_1),f(x_2))| < \epsilon,
    \end{align}
    or equivalently,
    \begin{align}
        |d_g(x_1,x_2) - d_{E/\Lambda^+}(f(x_1),f(x_2))| < \lambda^{-1}\epsilon.
    \end{align}
    Observe that the ball $\{\rho\leq(\lambda\epsilon)^{-\frac{2}{3}}\}\subset \mathbb{C}^n$ has radius $\lambda^\frac{1}{3}\epsilon^{-\frac{2}{3}}$ measured by $\lambda^2g$. So, its diameter is small when $\lambda$ is sufficiently small. For example, when $\lambda < \epsilon^5$, we have $\lambda^\frac{1}{3}\epsilon^{-\frac{2}{3}} < \epsilon$. Similarly, by Proposition \ref{proposition rho is comparable to cone radius}, its image under $f_\lambda$ is also of diameter smaller than $\epsilon$ if $\lambda$ is sufficiently small. Thus, we may assume that $\rho(x_i) > (\lambda\epsilon)^{-\frac{2}{3}}$

    Consider the case where $x_1,x_2\in S_{\alpha,c}$, without loss of generality, we may assume that $\xi_1(x_2) \geq \xi_1(x_1)$. By Lemma \ref{lemma singular region looks like a line}, we know that the difference between $d_g(x_1,x_2)$ and $\xi_1(x_2) - \xi_1(x_1)$ is smaller than
    \begin{align}
        C(\lambda\epsilon)^{-\frac{1+\alpha}{2}} + C\frac{\xi_1(x_2) - \xi_1(x_1)}{-\xi_1(x_2)}.
    \end{align}
    In $S_{\alpha,c}$, we have $-\xi_1\sim \rho$, so $\frac{-\xi_1(x_1)}{-\xi_1(x_2)}\sim \frac{\rho(x_1)}{\rho(x_2)}\leq \frac{\rho(x_2)^\frac{3}{2}}{\rho(x_2)} = \rho(x_2)^\frac{1}{2}\leq (\lambda\epsilon)^{-\frac{1}{2}}$. Here we have used the assumption that $\rho(x_i)\in [(\lambda\epsilon)^{-\frac{2}{3}}, (\lambda\epsilon)^{-1}]$.

    Note that since $x_i\in S_{\alpha,c}$, we have $|p_{nc}(\alpha_j)(x_i)|\leq C\rho(x_i)^{1+\alpha}$, so $r_j(x_i)\leq C\rho^\frac{1+\alpha}{2}$, and $||\sigma(x_2) - \sigma(x_1)| - (\xi_1(x_2) - \xi_1(x_1))| \leq C(\rho(x_1)^\alpha + \rho(x_2)^\alpha)$.

    Combining the above estimates, we know that
    \begin{align}
        |d_g(x_1,x_2) - d_{E/\Lambda^+}(f(x_1),f(x_2))| < C((\lambda\epsilon)^{-\frac{1+\alpha}{2}} + (\lambda\epsilon)^{-\frac{1}{2}} + (\lambda\epsilon)^{-\alpha}).
    \end{align}
    By choosing $\lambda_0$ sufficiently small, the right-hand side in the above inequality is smaller than $\lambda^{-1}\epsilon$. So $f_\lambda$ is an $\epsilon$-isometry when restricted to $S_{\alpha, c}$.

    Consider the case where $x_1,x_2\in R_{\alpha,c}$. Let $\gamma: [0,L] \rightarrow \mathbb{C}^n$ be the minimal geodesic with respect to $\lambda^2g$ that joins $x_1$ and $x_2$. If $\gamma$ is contained in $R_{\alpha,c}$, then by Proposition \ref{proposition g-g prime in regular region}, we have $|l_{\lambda^2g}(\gamma) - l_{\lambda^2g^\p}(\gamma)| < \epsilon$ if $\lambda$ is sufficiently small. It follows that
    \begin{align}
        d_{E/\Lambda^+}(f_\lambda(x_1),f_{\lambda}(x_2)) \leq d_{\lambda^2g^\p}(x_1,x_2) \leq l_{\lambda^2g^\p}(\gamma) \leq l_{\lambda^g}(\gamma) + \epsilon = d_{\lambda^2g}(x_1,x_2) + \epsilon.
    \end{align}

    If $\gamma$ is not contained in $R_{\alpha,c}$, then find $t_1 \in (0,L)$ such that $\gamma(t)\in R_{\alpha,c}$ for all $t\in (0,t_1)$ and $\gamma(t_1)\in S_{\alpha,c}$. Similarly, find $t_2 \in (0,L)$ such that $\gamma(t)\in R_{\alpha,c}$ for all $t\in (t_2,L)$ and $\gamma(t_2)\in S_{\alpha,c}$. Then if $\lambda$ is chosen small enough, we have
    \begin{align*}
        d_{E/\Lambda^+}(f_\lambda(x_1), f_\lambda(x_2)) \leq&
        d_{E/\Lambda^+}(f_\lambda(x_1), f_\lambda(\gamma(t_1))) +
        d_{E/\Lambda^+}(f_\lambda(\gamma(t_1)), f_\lambda(\gamma(t_2))) +\\
        &+d_{E/\Lambda^+}(f_\lambda(\gamma(t_2)), f_\lambda(x_2))\\
        \leq& d_{\lambda^2g}(x_1, \gamma(t_1)) + d_{\lambda^2g}(\gamma(t_1), \gamma(t_2)) + d_{\lambda^2g}(\gamma(t_2), x_2) + \epsilon\\
        =& d_{\lambda^2g}(x_1, x_2) + \epsilon.
    \end{align*}
    Thus, in any case we have $d_{E/\Lambda^+}(f_\lambda(x_1), f_\lambda(x_2)) \leq d_{\lambda^2g}(x_1, x_2) + \epsilon.$

    Conversly, let $\bar{\gamma}: [0,L]\rightarrow E/\Lambda^+$ be the minimal geodesic that joins $f_\lambda(x_1),f_\lambda(x_2)$ in $E/\Lambda^+$. Let $\gamma: [0,L]\rightarrow E$ be the horizontal lift of $\bar{\gamma}$ with respect to the submetry $f_\lambda$, such that $\gamma(0) = x_1$. Denote $x_2^\p = \gamma(L)$, then $f_\lambda(x_2) = f_\lambda(x_2^\p)$, and $d_{\lambda^2g^\p}(x_1,x_2^\p) = l_{\gamma^2g^\p}(\gamma)$. If $\gamma$ is contained in $R_{\alpha,c}$, then $|l_{\lambda^2g}(\gamma) - l_{\lambda^2g^\p}(\gamma)| < \frac{1}{2}\epsilon$ if $\lambda$ is sufficiently small. And by the proof of Proposition \ref{proposition asymptotic cone of g prime}, the $\Lambda^+$ orbit passing through $x_2$ has $g^\p$-diameter smaller than $\frac{1}{2}\epsilon$ if $\lambda$ is sufficiently small, hence its $g$-diameter is also smaller than $\frac{1}{2}\epsilon$ if $\lambda$ is sufficiently small. It follows that
    \begin{align}
        d_{\lambda^2g}(x_1,x_2) &\leq d_{\lambda^2g}(x_1,x_2^\p) + \frac{1}{2}\epsilon\\
        &\leq l_{\lambda^2g}(\gamma) + \frac{1}{2}\epsilon\\
        &\leq l_{\lambda^2g^\p}(\gamma) + \epsilon \\
        &= d_{E/\Lambda^+}(f_\lambda(x_1), f_\lambda(x_2)) + \epsilon.
    \end{align}

    If $\gamma$ is not contained in $R_{\alpha,c}$, then by dividing $\gamma$ into three segments as we have done before, we still have $d_{\lambda^2g}(x_1,x_2) \leq d_{E/\Lambda^+}(f_\lambda(x_1), f_\lambda(x_2)) + \epsilon$ when $\lambda$ is sufficiently small.

    In conclusion, we have shown that $f_\lambda$ is a $\epsilon$ isometry when restricted to $R_{\alpha, c}$.

    The case where $x_1\in S_{\alpha,c}$ and $x_2\in R_{\alpha,c}$ can be treated by a similar method as the case where $x_1,x_2\in R_{\alpha,c}$. In fact, it suffices to divide the geodesic into two segments. In summary, we have shown that $f_\lambda$ is an $\epsilon$-isometry if $\lambda<\lambda_0$ and $\lambda_0$ depends only on $\epsilon$.

    It remains to show that $f_\lambda$ is $\epsilon$-onto. Take any $y\in E/\Lambda^+$ with $|y|<\epsilon^{-1}$. Recall that the image $f_\lambda(M^0)$ is determined by $\frac{1}{4}\sum_{j=1}^{l-1}\frac{1}{\alpha_l-\alpha_j}r_j^2>-\lambda\sigma$ and $r_j>0$. So for $\lambda$ sufficiently small, there exists $x\in M^0$ such that $d_{E/\Lambda^+}(y,f_\lambda(x))<\epsilon.$ It follows that $|f_\lambda(x)| < \epsilon^{-1} + \epsilon$. By Proposition \ref{proposition rho is comparable to cone radius}, there exists $C_0>0$ independent of $\epsilon$ such that $\rho(x) \leq C_0\lambda^{-1}\epsilon^{-1}$. Since we have proved that $f_\lambda$ is $\epsilon$-isometric, for $\lambda$ sufficiently small, we have $|\lambda\rho(x) - |f_\lambda(x)|| < \epsilon$. Hence, we have $\lambda\rho(x) < \epsilon^{-1} + 2\epsilon$.

    Let $x^\p\in \mathbb{C}^n$ such that $d_{\lambda^2g}(x,x^\p) < 2\epsilon$ and $\lambda\rho(x^\p) < \epsilon$, then
    \begin{align}
       d_{E/\Lambda^+}(y,f_\lambda(x^\p)) &< d_{E/\Lambda^+}(y,f_\lambda(x)) + d_{E/\Lambda^+}(f_\lambda(x^\p), f_\lambda(x^\p)) \\
       &< \epsilon + d_{\lambda^2g}(x,x^\p) + \epsilon\\
       &< 4\epsilon
    \end{align}
    Thus, we have shown that $f_\lambda$ is also $\epsilon$-onto.
    In conclusion, $E/\Lambda^+$ is an asymptotic cone of $(\mathbb{C}^n, g)$. Moreover, since $(\mathbb{C}^n, g)$ is complete, our proof implies that $E/\Lambda^+$ is the unique asymptotic cone of $(\mathbb{C}^n, g)$.
\end{proof}

The dimension of the asymptotic cone $E/\Lambda^+$ is $2n-1-\dim\Lambda$. By Proposition \ref{proposition dimension of Lambda}, it is $2n - \dim_\mathbb{Q}\op{Span}_\mathbb{Q}\{1,\tau_1,\dots,\tau_{l-1}\}$.

Note that in \cite[Lemma 5.8]{apostolov2023hamiltonian} it is proved that the volume growth of $(\mathbb{C}^n,g)$ is of order $r^{2n-1}$. So, if $\tau$ is rational, then the order of volume growth matches the dimension of the asymptotic cone. However, if not, then the order of volume growth is strictly larger than the dimension of the asymptotic cone.

\begin{example}
    Consider the special case where $l=2$, then have to set $d_1 = n-2$. In this case $\tau_1 = -\frac{1}{n-1}$, so $\Lambda = \mathbb{Z}_{n-1}$. By Theorem \ref{theorem asymptotic cone of g}, the asymptotic cone is $(\mathbb{C}^{n-1}/\mathbb{Z}_{n-1})\times \mathbb{R}$. If $a=0$ and $n=2$, then it is known in \cite{apostolov2023hamiltonian} that we will get the Taub-NUT metric, whose asymptotic cone is $\mathbb{R}^3$.
\end{example}

\begin{example}
    Consider the case where $l=3$, without loss of generality, we may assume $\alpha_1=0$ and $\alpha_2=1$. Then it follows that $\tau_1 = -\frac{1}{d_1 + 1}(\alpha_3-1)$ and $\tau_2 = -\frac{1}{d_2 + 1}\alpha_3$. If $\alpha_3\in \mathbb{Q}$, then $\dim\Lambda = 0$ and the asymptotic cone is of dimension $2n-1$. If $\alpha_3\notin\mathbb{Q}$, then $\dim\Lambda = 1$ and the asymptotic cone is of dimension $2n-2$.
\end{example}

\section{The special case where $l=2$ and $a=0$}\label{section The special case}
In this section we will study in detail the special case where $l=2$ and $a=0$. If $n>2$, then it is a generalization of the Taub-NUT metric (which is an ALF gravitational instanton) to higher dimensions. Indeed, we will show that this metric is ALF in a certain sense.

\subsection{Description of the metric}
When $l=2$, we have $d_1 = n-2$. Without loss of generality, we may assume that $\alpha_1=0$, $\alpha_2=1$. Consequently, we have $\mathring{D}= (-\infty, 0)\times(1, +\infty)$. Let $(\xi_1,\xi_2)\in \mathring{D}$ be its coordinates, then $p_c(t) = t^{n-2}$, $p_{nc}(t) = (t-\xi_1)(t-\xi_2)$, $P(t) = F_1(t) = t^{n-1}$. Under the assumption that $a=0$, we have $q(t) = n-1$ and $F_2(t) = t^{n-1} - 1$.

In this special case, $v_1 = (0,-2)$ and $v_2 = \frac{2}{n-1}(-1,1)$. Now $P$ is the principal $\mathbb{T}^2$-bundle associated with the complex vector bundle $O(-1)\oplus \mathbb{C}$ over $\mathbb{CP}^{n-2}$. Here, $O(-1)$ is the tautological bundle and $\mathbb{C}$ is the trivial bundle. Recall that $T_1,T_2$ are the vector fields on $P$ corresponding to the rotation in $O(-1)$ and the trivial bundle $\mathbb{C}$. Let $\eta_1,\eta_2$ be the connection $1$-form associated with $T_1,T_2$ such that $\eta_i(T_j) = \delta_{ij}$ and curvature $d\eta_1=\check{\omega}_1^0$, $d\eta_2=0$.

Recall that $T_1,T_2$ correspond to $v_1,v_2$ and $K_1,K_2$ correspond to $e_1,e_2$, so we have $K_1 = -\frac{n-1}{2}T_2 - \frac{1}{2}T_1$ and $K_2 = -\frac{1}{2}T_1$. Since $\theta_1,\theta_2$ is the dual basis to $K_1,K_2$, we have
\begin{align}\label{theta1}
\theta_1 &= -\frac{2}{n-1}\eta_2,\\ \label{theta2}
\theta_2 &= -2\eta_1 + \frac{2}{n-1}\eta_2.
\end{align}
It follows that $d\theta_1=0$ and $d\theta_2 = -2\check{\omega}_1^0 = -\check{\omega}_1$.

Now the K\"ahler metric $(g,\omega)$ on $M^0 = \mathring{D}\times P$ is given by
\begin{align}\label{metric of Apostolov-Cifarelli}
g = &-\xi_1\xi_2\check{g}_1 + \frac{\xi_2 - \xi_1}{-\xi_1}(d\xi_1)^2 + \frac{\xi_2^{n-2}(\xi_2-\xi_1)}{\xi_2^{n-1}-1}(d\xi_2)^2 +\\ \label{end of g}
    &+ \frac{-\xi_1}{\xi_2-\xi_1}(\theta_1 + \xi_2\theta_2)^2 + \frac{\xi_2^{n-1}-1}{\xi_2^{n-2}(\xi_2-\xi_1)}(\theta_1 + \xi_1\theta_2)^2,\\
\omega = &-\xi_1\xi_2\check{\omega}_1 + d\sigma_1\wedge\theta_1 + d\sigma_2\wedge\theta_2, \label{end of formulas of Apostolov-Cifarelli}
\end{align}
where $\sigma_1 = \xi_1 + \xi_2$, $\sigma_2 = \xi_1\xi_2$. Its extension to $\mathbb{C}^n$ is the Apostolov-Cifarelli metric (of the Taub-NUT type) and we still denote it by $(g,\omega)$.

In this case, we have
\begin{align}
g^\p = &-\xi_1\xi_2\check{g}_1 + \frac{\xi_2-\xi_1}{-\xi_1}(d\xi_1)^2 + \frac{\xi_2-\xi_1}{\xi_2}(d\xi_2)^2 + \\
       &+\frac{-\xi_1}{\xi_2-\xi_1}(\theta_1 + \xi_2\theta_2)^2 + \frac{\xi_2}{\xi_2-\xi_1}(\theta_1 + \xi_1\theta_2)^2.
\end{align}
Note that $\tau_1 = -\frac{1}{n-1}$, then by Proposition \ref{proposition isometric embedding of g prime}, we know that $(M^0,g^\p)$ can be isometrically embedded into $E = ((\mathbb{C}^{n-1}\times \mathbb{R})/\mathbb{Z})\times \mathbb{R} = ((\mathbb{C}^{n-1}\times \mathbb{S}^1)/\mathbb{Z}_{n-1})\times \mathbb{R}$ where $\mathbb{Z}_{n-1}$ acts on $\mathbb{C}^{n-1}$ and $\mathbb{S}^1$ by rotation. By Theorem \ref{theorem asymptotic cone of g}, its asymptotic cone is $(\mathbb{C}^{n-1}/\mathbb{Z}_{n-1})\times \mathbb{R}$.

\subsection{The $SOB(2n-1)$ property}
\begin{proposition}\label{proposition volume growth of g}
The volume growth of $g$ is of order $2n-1$. More precisely, there exists $C>0$ such that for any $R>C$, we have $C^{-1}R^{2n-1} \leq Vol(g, B(g,R)) \leq CR^{2n-1}$.
\end{proposition}
For the proof of Proposition \ref{proposition volume growth of g}, we refer to \cite[Lemma 5.8]{apostolov2023hamiltonian}. In fact, it is true for any choice of $l\geq 2$ and $a\geq 0$.

We will show that the Apostolov-Cifareli metric of the Taub-NUT type is $SOB(2n-1)$ in the following sense:
    \begin{definition}\label{definition of SOB}
        A complete non-compact Riemannian manifold $(N,g)$ of dimension $m>2$ is called $SOB(\beta)$, ($\beta\in\mathbb{R_+}$) if there exist $x_0\in N$ and $C\geq 1$ such that
        \begin{itemize}
            \item Let $A(x_0,s,t)=\{s\leq r(x)\leq t\}$ be the annulus, then for sufficiently large $D>0$, any two points $m_1,m_2\in N$ with $r(m_i)=D$ can be joined by a curve of length at most $CD$, lying in the annulus $A(x_0,C^{-1}D,CD)$,
            \item $Vol(g,B(x_0,s))\leq Cs^\beta$ for all $s\geq C$,
            \item $Vol(g,B(x,(1-\frac{1}{C})r(x)))\geq \frac{1}{C}r(x)^\beta$,
            \item $\op{Ric}(x)\geq -Cr(x)^{-2}$,
        \end{itemize}
        if $r(x)=d(x_0,x)\geq C$.
    \end{definition}
\begin{proposition}\label{proposition SOB(2n-1)}
    If $l=2$ and $a=0$, then the Apostolov-Cifareli metric $(\mathbb{C}^n,g)$ of the Taub-NUT type is $SOB(2n-1)$.
\end{proposition}
\begin{proof}
    We will take $x_0$ to be the origin of $\mathbb{C}^n$, so $r(x) = \rho(x)$.

    We start with the first condition in Definition \ref{definition of SOB}. Fix $0<\alpha < 1$ and $c>0$, if both $m_1,m_2$ are in $R_{\alpha,c}$, then we can apply Proposition \ref{proposition g-g prime in regular region} and the fact that the Euclidean metric of $E$ satisfies this condition. If one of $m_1,m_2$ is in $S_{\alpha,c}$, then by Lemma \ref{lemma connecting singular to regular}, it can be conected to $R_{\alpha,c}$ by a curve of length much shorter than $D$ since $\frac{1+\alpha}{2} < 1$.

    The second condition is a consequence of Proposition \ref{proposition volume growth of g}.

    As for the third condition, if $x\in R_{\alpha,c}$, then again we can apply Proposition \ref{proposition g-g prime in regular region} to prove the lower bound of volume growth. If $x\in S_{\alpha,c}$, then by Lemma \ref{lemma connecting singular to regular} again, there is a point $x^\p$ in $R_{\alpha,c}$ that lies in $B(x,C\rho(x)^{\frac{1+\alpha}{2}})$. If $C$ is sufficiently large, large and $\rho(x)>C$, then $B(x^\p,\frac{1}{2}\rho(x)) \subset B(x,(1-\frac{1}{C})r(x))$, hence we will get a lower bound of volume growth.

    Finally, since $a=0$, the metric is Ricci-flat. This proves the last condition.
\end{proof}
The above proof is based on the study of the metric of Apostolov and Cifarelli. In fact, we have the following general result.

\begin{proposition}\label{proposition SOB general}
    Let $(M,g)$ be a connected complete nonncompact Riemannian manifold of dimension $n$ with nonnegative Ricci curvature. Let $x_0 \in M$ and assume that there exist $A,B,\beta>0$ such that $Ar^\beta \leq Vol(B(x_0,r)) \leq Br^\beta$ for any $r\geq 1$. Then $(M,g)$ is $SOB(\beta)$.
\end{proposition}
\begin{proof}
    By assumption, the conditions concerning the Ricci curvature and upper bound of the volume of balls are satisfied. If $C\geq 2$, then by Bishop-Gromov inequality, for any $x\in M$ with $r(x)\geq C$ we have
    \begin{align}
        Vol(B(x,(1-\frac{1}{C})r(x))) &\geq (\frac{1-\frac{1}{C}}{2})^nVol(B(x,2r(x))) \\
        &\geq \frac{1}{4^n}Vol(B(x_0,r(x)))\\
        &\geq \frac{A}{4^n}r(x)^\beta.
    \end{align}
    This proves the condition concerning the lower bound of the volume of balls.

    Finally, to prove the relatively connected annuli (RCA) condition, we will apply a result of Minerbe. By Bishop-Gromov inequality again, we know that $Vol(B(x,2t)) \leq 2^nVol(B(x,t))$ for any $x\in M$ and $t>0$. The result of \cite{buser1982} establishes the $L^p$ Poincar\'e inequality for any $1\leq p <+\infty$. Now we can apply \cite[Proposition 2.8]{Minerbe2009} to show the RCA condition.
\end{proof}
Consequently, for any $l\geq 2$ and any choice of $d_j$, as long as $a=0$, the Calabi-Yau metric of Taub-NUT type on $\mathbb{C}^n$ constructed by Apostolov and Cifarelli is $SOB(2n-1)$.

\subsection{A K\"ahler potential}

\begin{proposition}\label{proposition kahler potential of g}
For any constant $C$, the following function defines a K\"ahler potential of $g$:
\begin{align}\label{formula of potential of g}
H = \frac{1}{2}\xi_1^2 - \xi_1 + \frac{1}{2}\xi_2^2-\xi_2 + \int_0^{\xi_2}\frac{1}{1+t+\dots + t^{n-2}}dt + C.
\end{align}
More precisely, we have $dd^cH=\omega$.
\end{proposition}

The last proposition is deduced from the formula of K"ahler potential for K"ahler metric that admits Hamiltonian 2-forms given in \cite[Theorem 1]{ACGI}. In fact, one starts with the K\"ahler potential given by formula (65) of \cite{ACGI} and then subtracts it by a pluriharmonic function $u_1$ given by formula (70) of \cite{ACGI}.

\begin{proposition}\label{proposition properties of potential}
By fixing a large enough constant $C$, the potential $H$ given by \eqref{formula of potential of g} satisfies the following properties:
\begin{itemize}
  \item The function $H$ is positive;
  \item The function $H$ is comparable to $\rho^2$ outside a compact set;
  \item The function $H$ satisfies $dH \wedge d^cH \leq C^\p Hdd^cH$ for some $C^\p > 0$.
\end{itemize}
\end{proposition}
\begin{proof}
We choose $C >0$ such that $\frac{1}{2}x^2 - x + \frac{1}{2}C > \frac{1}{4}x^2$ for any $x\in \mathbb{R}$, then
\begin{align}
H &= (\frac{1}{2}\xi_1^2 -\xi_1 + \frac{1}{2}C) + (\frac{1}{2}\xi_2^2 -\xi_2 + \frac{1}{2}C) + \int_0^{\xi_2}\frac{1}{1+t+\dots + t^{n-2}}dt \\
  &> \frac{1}{4}(\xi_1^2 + \xi_2^2 ) \geq \frac{1}{4}.
\end{align}
Regarding the second property, recall Proposition \ref{proposition rho is comparable to xi_l-xi_1} that outside a compact set $\rho^2$ is proportional to $(\xi_2 - \xi_1)^2$, which is in turn proportional to $\xi_1^2 + \xi_2^2$ since $\xi_1\leq 0$, $\xi_2\geq 1$. Note that the integration in the formula of $H$ is bounded from above by $\xi_2$, we have
\begin{align}
H &\leq \frac{1}{2}\xi_1^2 - \xi_1 + \frac{1}{2}\xi_2^2 + C \\
  &\leq \frac{1}{2}\xi_1^2 + \frac{1}{2}\xi_1^2 + \frac{1}{2} + \frac{1}{2}\xi_2^2 + C \\
  &\leq \xi_1^2 + \xi_2^2 + C + \frac{1}{2} \\
  &\leq (C+\frac{3}{2})(\xi_1^2 + \xi_2^2).
\end{align}
Thus, $H$ is comparable to $\rho^2$.

Finally, to prove the last inequality, it suffices to show that $(dH)^2 \leq C^\p H g$ for some $C^\p>0$. Now $dH = (\xi_1-1)d\xi_1 + \frac{\xi_2^{n-1}}{\xi_2^{n-2} + \dots + \xi_2 + 1}d\xi_2$. Observing that the coefficients of $d\xi_i$ in $dH$ are bounded from above by a multiple of $\sqrt{H}$, while the coefficients of $(d\xi_i)^2$ of $g$ are bounded below by $1$, we get $(dH)^2 \leq C^\p H g$ for some $C^\p>0$.
\end{proof}

\subsection{Estimation of the Riemannian curvature}
The aim of this subsection is to prove the following proposition:

\begin{proposition}\label{proposition better estimation of Rm of g}
There exists $C > 0$ such that for any $x\in \mathbb{C}^n$, we have $|\op{Rm}_g(x)| \leq \frac{C}{1 + \rho(x)}$.
\end{proposition}

We will prove this with the help of the following weaker version of the result of \cite[Theorem 1.1]{Naber-Zhang}:
\begin{theorem}\label{theorem simple version of Naber-Zhang}
Let $p$ be a point of the Riemannian manifold $(M,g)$ of dimension $n$, assume that $B_2(p)$ has a compact closure in $B_4(p)$. Assume that $(M^n, g, p)$ satisfies $|\op{Ric}|\leq n-1$, then there exist positive constants $\delta, w_0, c_0$ depending only on $n$, such that if
\begin{align}
d_{GH}(B_2(p), B_2(0^k))<\delta,
\end{align}
where $0^k\in\mathbb{R}^k$ is the origin of the standard Euclidean Riemannian manifold $\mathbb{R}^k$ ($0\leq k\leq n$), then $\Gamma_\delta(p) = \op{Image}[\pi_1(B_\delta(p))\rightarrow \pi_1(B_2(p))]$ is $(w_0,n-k)$-nilpotent with $\op{rank}(\Gamma_\delta(p))\leq n-k$, and if equality holds then for each $q\in B_1(p)$ we have the conjugate radius bound
\begin{align}
\op{ConjRad}(q) \geq c_0 >0.
\end{align}
In particular, if $M^n$ is Einstein, then we have
\begin{align}
\sup_{B_1(p)}|\op{Rm}| \leq C(n).
\end{align}
\end{theorem}
\begin{remark}
The result of \cite[Theorem 1.1]{Naber-Zhang} is stronger than the above statement, in fact it allows us to replace $\mathbb{R}^k$ by any product $\mathbb{R}^{k-l}\times Z^l$ for $l\leq 3$, where $Z$ is a Ricci-limit space of dimension $l$ in the sense of \cite{Colding-Naber}, and the constants depend on the ball of radius $2$ of $\mathbb{R}^{k-l}\times Z^l$. Moreover, the statement admits a converse.
\end{remark}
\begin{remark}
We will apply Theorem \ref{theorem simple version of Naber-Zhang} to cases $k=2n$ and $k=2n-1$. Here we note that if a group $\Gamma$ is $(w_0, 1)$-nilpotent with $\op{rank}(\Gamma)\leq 1$, then to prove that $\op{rank}(\Gamma)= 1$, it suffices to show that $\Gamma$ is infinite (see, for example, \cite[Section 2.4.1]{Naber-Zhang}).
\end{remark}
Now, let us give the proof of Proposition \ref{proposition better estimation of Rm of g}.

\begin{proof}[Proof of Proposition \ref{proposition better estimation of Rm of g}:]
We will proceed by contradiction. Suppose that the Riemannian curvature bound fails. Then for any sequence $C_i\rightarrow +\infty$, there exists $x_i\in \mathbb{C}^n$ such that
\begin{align}
|\op{Rm}_g(x_i)|\geq \frac{C_i}{1 + \rho(x_i)}.
\end{align}
By selecting a subsequence, we may assume that $\rho_i = \rho(x_i) \rightarrow +\infty$. We may choose a sequence of real numbers $m_i>0$ such that $m_i\rightarrow +\infty$, $\frac{C_i}{m_i^2}\rightarrow +\infty$ and $\frac{\rho_i}{m_i^2} \rightarrow +\infty$. We define $\lambda_i = m_i\rho_i^{-\frac{1}{2}}$, then $\lambda_i\rightarrow 0$ and $\lambda_i^2\rho_i = m_i^2 \rightarrow +\infty$ as $i\rightarrow +\infty$.

For any $0<\beta <1$ and $c>0$, denote as before $R_{\beta,c} = \{\xi_2 \geq c(\xi_2-\xi_1)^\beta\} \subset \mathbb{C}^n$.

First, we consider the case where $x_i$ lies in the regular region $R_{\beta,c}$. Then by Proposition \ref{proposition isometric embedding of g prime} and Proposition \ref{proposition g-g prime in regular region}, we know that the metric tensor of $B_{\lambda_i^2\tilde{g}}(\tilde{x}_i,2)$ $C^0$-converges to the ball of radius $2$ in $\mathbb{R}^{2n}$, where $\tilde{g}$ is the pull back of $g$ from $((\mathbb{C}^{n-1}\times\mathbb{S}^1)/\mathbb{Z}_{n-1})\times \mathbb{R}$ to its universal covering $\mathbb{R}^{2n}$. So for sufficiently large $i$, we have
\begin{align}
d_{GH}(B_{\lambda_i^2\tilde{g}}(\tilde{x}_i,2), B_{\mathbb{R}^{2n}}(0,2)) < \delta
\end{align}
and $\op{rank}(\Gamma_\delta(\tilde{x}_i)) = 0$. Applying Theorem \ref{theorem simple version of Naber-Zhang}, we have
\begin{align}
|\op{Rm}_{\lambda_i^2g}(x_i)| = |\op{Rm}_{\lambda_i^2\tilde{g}}(\tilde{x}_i)| \leq C(n).
\end{align}

Next, we consider the case where $x_i$ is in the complement of $R_{\beta,c}$. So, by the definition of $S_{\beta,c}$, we have $-\xi_1(x_i)$ comparable to $\rho_i$ and $\xi_2(x_i)\leq C\rho_i^\beta$. Note that the distance $\tilde{\rho}(x_i)$ of $x_i$ measured by $\lambda_i^2g$ is $\lambda_i\rho_i\rightarrow +\infty$, so the ball $B_{\lambda_i^2g}(x_i,2)$ is far from the origin for sufficiently large $i$. Now we have
\begin{align}
\lambda_i^2g = &-\lambda_i^2\xi_1\xi_2\check{g}_1 + \lambda_i^2\frac{\xi_2 - \xi_1}{-\xi_1}(d\xi_1)^2 + \lambda_i^2\frac{\xi_2^{n-2}(\xi_2-\xi_1)}{\xi_2^{n-1}-1}(d\xi_2)^2 +\\
    &+ \lambda_i^2\frac{-\xi_1}{\xi_2-\xi_1}(\theta_1 + \xi_2\theta_2)^2 + \lambda_i^2\frac{\xi_2^{n-1}-1}{\xi_2^{n-2}(\xi_2-\xi_1)}(\theta_1 + \xi_1\theta_2)^2.
\end{align}
Examining the second term of the above formula, and note that $2\frac{\xi_2 - \xi_1}{-\xi_1} \geq 1$, we know that for any $x_i^\p\in B_{\lambda_i^2g}(x_i,2)$, we have
\begin{align}
|\xi_1(x_i^\p) - \xi_1(x_i)| \leq 2 \lambda_i^{-1}.
\end{align}
By the choice of $m_i$, we know that $\lambda_i^{-1}$ is much smaller compared to $\rho_i$, consequently $|\frac{\xi_1(x_i^\p) - \xi_1(x_i)}{\xi_1(x_i)}|\rightarrow 0$.

Introduce a new coordinate $u\geq 0$  by $u^2 = \xi_2 - 1$, then the third term in the formula of $\lambda_i^2g$ can be written as
\begin{align}
\lambda_i^2\frac{\xi_2^{n-2}(\xi_2-\xi_1)}{\xi_2^{n-1}-1}(d\xi_2)^2 = 4\lambda_i^2(\xi_2-\xi_1)\frac{\xi_2^{n-2}}{1+\xi_2+\dots+\xi_2^{n-2}}(du)^2.
\end{align}
The function $\frac{\xi_2^{n-2}}{1+\xi_2+\dots+\xi_2^{n-2}}$ is bounded from above and from below by positive numbers when $\xi_2\geq 1$. And note that $\xi_2-\xi_1$ is comparable to $\rho$. On the ball $B_{\lambda_i^2g}(x_i,2)$, for any $x_i^\p\in B_{\lambda_i^2g}(x_i,2)$, we have $|\rho(x_i^\p) - \rho_i| \leq 2\lambda_i^{-1}$, which is much smaller than $\rho_i$. So we know that for any $x_i^\p\in B_{\lambda_i^2g}(x_i,2)$, we have
\begin{align}
|u(x_i^\p) - u(x_i)| \leq C\lambda_i^{-1}\rho_i^{-\frac{1}{2}} = Cm_i^{-1}\rightarrow 0.
\end{align}
It follows that the function $u$ (hence the function $\xi_2$) is $C^0$-close to a constant function on the ball $B_{\lambda_i^2g}(x_i,2)$. In particular $|\frac{\xi_2(x_i^\p) - \xi_2(x_i)}{\xi_2(x_i)}|\rightarrow 0$. Knowing the above estimations of the range of $\xi_1$ and $\xi_2$ on the ball, we deduce that by replacing $c$ with another constant $c^\p$, we may assume that the ball $B_{\lambda_i^2g}(x_i,2)$ is entirely contained in the singular region $S_{\beta,c^\p}$, and the function $\rho$ is comparable to $\rho_i$ on the ball.

Now we look at the first term of the formula of $\lambda_i^2g$, by the above discussion we see that $-\lambda_i^2\xi_1\xi_2\check{g}_1$ is $C^0$-close to $-\lambda_i^2\xi_1(x_i)\xi_2(x_i)\check{g}_1$ measured by $\lambda_i^2g$. Since $-\lambda_i\xi_1(x_i)\xi_2(x_i) \geq C\lambda_i\rho_i = Cm_i\rightarrow +\infty$, we know that the first term $C^0$ converges to the Euclidean metric $\mathbb{R}^{2n-4}$ (the tangent cone at any point of a smooth manifold is a Euclidean space). Moreover, recall that the formula of $g$ is defined on $M^0 = \mathring{D}\times P$, we know that under the projection $M^0\rightarrow P \rightarrow \mathbb{CP}^{n-2}$, the image of $B_{\lambda_i^2g}(x_i,2)$ is contained in a small ball of radius smaller than $\frac{C}{\sqrt{-\lambda_i^2\xi_1(x_i)\xi_2(x_i)}} \leq \frac{C}{\sqrt{m_i}}\rightarrow 0$ on $\mathbb{CP}^{n-2}$ measured by $\check{g}_1$. Let $U$ be a simply connected subset of $\mathbb{CP}^{n-2}$ containing the image of $B_{\lambda_i^2g}(x_i,2)$ under the projection, then our formula of $g$ can be seen as defined on $\mathring{D}\times P|_U$, here $P|_U$ means the restriction of the $\mathbb{T}^2$-principal fibration to the subset $U$ of the base $\mathbb{CP}^{n-2}$.

Similarly we have
\begin{align}
|\lambda_i^2\frac{\xi_2-\xi_1}{-\xi_1}(d\xi_1)^2 - \lambda_i^2(d\xi_1)^2|_{\lambda_i^2g} \rightarrow 0,
\end{align}
as $i\rightarrow +\infty$. So, the second term of the formula of $\lambda_i^2g$ converges to the Euclidean metric of $\mathbb{R}$.

For the last two terms of the formula of $\lambda_i^2g$, we have
\begin{align}
&\frac{-\xi_1}{\xi_2-\xi_1}(\theta_1 + \xi_2\theta_2)^2 + \frac{\xi_2^{n-1}-1}{\xi_2^{n-2}(\xi_2-\xi_1)}(\theta_1 + \xi_1\theta_2)^2 \\ \label{decomposition of last two terms1}
=& (1 - \frac{1}{\xi_2^{n-2}(\xi_2-\xi_1)})(\theta_1 + \xi_2\theta_2)^2 -2\frac{\xi_2^{n-1}-1}{\xi_2^{n-2}}(\theta_1 + \xi_2\theta_2)\theta_2 + \\ \label{decomposition of last two terms2}&+\frac{\xi_2^{n-1}-1}{\xi_2^{n-2}}(\xi_2-\xi_1)\theta_2^2.
\end{align}
We have $\xi_2^{n-2}(\xi_2-\xi_1)\geq \xi_2-\xi_1 \geq C\rho_i\rightarrow +\infty$, and $\frac{\xi_2^{n-1}-1}{\xi_2^{n-2}}$ is comparable to $\xi_2-1$ which is bounded from above by $C\rho_i^\beta$. It follows that the middle term of \eqref{decomposition of last two terms1}-\eqref{decomposition of last two terms2} is much smaller in comparison to the other two terms as $i\rightarrow +\infty$. So, the $C^0$ difference between $\frac{-\xi_1}{\xi_2-\xi_1}(\theta_1 + \xi_2\theta_2)^2 + \frac{\xi_2^{n-1}-1}{\xi_2^{n-2}(\xi_2-\xi_1)}(\theta_1 + \xi_1\theta_2)^2$ and $(\theta_1 + \xi_2\theta_2)^2 + (\xi_2-\xi_1)\frac{1 + \xi_2 + \dots + \xi_2^{n-2}}{\xi_2^{n-2}}(\xi_2-1)\theta_2^2$ measured by $g$ converges to $0$ on the ball $B_{\lambda_i^2g}(x_i,2)$ as $i\rightarrow +\infty$.

Furthermore, since $U$ is simply connected, we may write $\theta_2 = \theta_2^\p + \zeta$, where $\theta_2$ is defined on $P|_U$ and its restriction to the $\mathbb{T}^2$-fibers is the same as $\theta_2$ and $d\theta_2^\p = 0$, while $\zeta$ is a $1$-form on $\mathbb{CP}^{n-2}$ such that $d\zeta = \hat{\omega}_1$. In fact, we may assume that on the ball $B_{\lambda_i^2g}(x_i,2)$, we have $|\zeta|_{\check{g}_1} \leq \frac{C}{\sqrt{m_i}}$. Here we play the following trick: Let $Z_0,\dots,Z_n$ be the homogeneous coordinates of $\mathbb{CP}^n$, then in the local chart $\{Z_0\neq 0\}$ with local coordinates $z_i = \frac{Z_i}{Z_0}$ ($i = 1,\dots, n$), a K\"ahler potential of the Fubini-Study metric is $\log(|z|^2 + 1)$, where $|z|^2 = |z_1|^2+\dots+|z_n|^2$, and note that $|d^c\log(|z|^2 + 1)| = O(|z|)$ measured by the Fubini-Study metric as $|z|\rightarrow 0$. By the symmetry of $\mathbb{CP}^{n}$, on any small ball of radius $\epsilon>0$ of $\mathbb{CP}^{n}$, there exists a $1$-form $\zeta$ such that $d\zeta$ is the K\"ahler form of the Fubini-Study metric and $|\zeta|\leq C\epsilon$. As a consequence of this trick, if we replace all the $\theta_2$ by $\theta_2^\p$ in the formula of $g$, the difference of the metric caused by this change is $C^0$-small measure by $g$.

In conclusion, on the ball $B_{\lambda_i^2g}(x_i,2)$, as $i\rightarrow+\infty$ the metric $\lambda_i^2g$ is $C^0$-close to
\begin{align}
&-\lambda_i^2\xi_1(x_i)\xi_2(x_i)\check{g}_1 + \lambda_i^2(d\xi_1)^2 + \lambda_i^2(\theta_1 + \xi_2\theta_2^\p)^2 +\\
 &+\lambda_i^2(\xi_2-\xi_1)\left[ 4\frac{\xi_2^{n-2}}{1+\xi_2+\dots+\xi_2^{n-2}}(du)^2 + \right.\\
 &+\left. \frac{1+\xi_2+\dots+\xi_2^{n-2}}{\xi_2^{n-2}}u^2(\theta_2^\p)^2 \right].
\end{align}

Recall that $|\frac{\xi_1(x_i^\p) - \xi_1(x_i)}{\xi_1(x_i)}|\rightarrow 0$ and $|\frac{\xi_2(x_i^\p) - \xi_2(x_i)}{\xi_2(x_i)}|\rightarrow 0$ for any $x_i^\p \in B_{\lambda_i^2g}(x_i,2)$, we may replace $\xi_1,\xi_2$ in the above formula by the constants $\xi_1(x_i),\xi_2(x_i)$. That is to say, on the ball $B_{\lambda_i^2g}(x_i,2)$, as $i\rightarrow+\infty$ the metric $\lambda_i^2g$ is $C^0$-close to
\begin{align}
&-\lambda_i^2\xi_1(x_i)\xi_2(x_i)\check{g}_1 + \lambda_i^2(d\xi_1)^2 + \lambda_i^2(\theta_1 + \xi_2(x_i)\theta_2^\p)^2 +\\
&+\lambda_i^2(\xi_2(x_i)-\xi_1(x_i))\left[4\frac{\xi_2(x_i)^{n-2}}{1+\xi_2(x_i)+\dots+\xi_2(x_i)^{n-2}}(du)^2 + \right.\\ &\left.+\frac{1+\xi_2(x_i)+\dots+\xi_2(x_i)^{n-2}}{\xi_2(x_i)^{n-2}}u^2(\theta_2^\p)^2\right].
\end{align}

Observe that the tensor in the braket
\begin{align}\label{last term of approx of lambda^2g}
4\frac{\xi_2(x_i)^{n-2}}{1+\xi_2(x_i)+\dots+\xi_2(x_i)^{n-2}}(du)^2 + \frac{1+\xi_2(x_i)+\dots+\xi_2(x_i)^{n-2}}{\xi_2(x_i)^{n-2}}u^2(\theta_2^\p)^2
\end{align}
is a multiple of the metric tensor of a flat cone of dimension $2$ with certain angle at the vertex.

Now choose a sequence of positive real numbers $n_i$ such that $n_i\rightarrow +\infty$ and $\frac{m_i}{n_i}\rightarrow +\infty$. Then we consider the following two subcases.

In the first subcase, we assume that $u(x_i)\geq n_i^{-1}$. Recall that $|u(x_i^\p) - u(x_i)| \leq Cm_i^{-1}$, which is in turn much smaller than $n_i^{-1}$. So we deduce that in this subcase the function $u$ is strictly positive on the ball $B_{\lambda_i^2g}(x_i,2)$. It follows that no matter what angle the 2-dimensional cone is, the term \eqref{last term of approx of lambda^2g} converges to $g_{\mathbb{R}^2}$.

Recall that the $\mathbb{T}^2$-fibration $P$ is trivial on $U$, so $P|_U$ is diffeomorphic to $U\times\mathbb{T}^2$. Let $U\times \mathbb{R}^2$ be its universal covering, and let $\tilde{g}$ be the pullback of $g$ to $\tilde{M}^0 = \mathring{D}\times U\times \mathbb{R}^2$. Now $\tilde{M}^0$ is simply connected, so there exist real functions $t_1,t_2$ such that $dt_1=\theta_1$ and $dt_2 = \theta_2^\p$. It follows that on the ball $B_{\lambda_i^2\tilde{g}}(\tilde{x}_i,2)$ (where $\tilde{x}_i$ is any pullback of $x_i$), as $i\rightarrow +\infty$ the metric $\lambda_i^2\tilde{g}$ is $C^0$-close to

\begin{align}
&-\lambda_i^2\xi_1(x_i)\xi_2(x_i)\check{g}_1 + \lambda_i^2(d\xi_1)^2 + \lambda_i^2(dt_1 + \xi_2(x_i)dt_2)^2 +\\
&+\lambda_i^2(\xi_2(x_i)-\xi_1(x_i))\left[4\frac{\xi_2(x_i)^{n-2}}{1+\xi_2(x_i)+\dots+\xi_2(x_i)^{n-2}}(du)^2 + \right.\\
&\left.+\frac{1+\xi_2(x_i)+\dots+\xi_2(x_i)^{n-2}}{\xi_2(x_i)^{n-2}}u^2(dt_2)^2\right].
\end{align}

So as $i\rightarrow +\infty$, the above metric $C^0$-converges to $g_{\mathbb{R}^{2n-4}} + g_\mathbb{R} + g_\mathbb{R} + g_{\mathbb{R}^2} = g_{\mathbb{R}^{2n}}$ and it follows that for sufficiently large $i$, we have
\begin{align}
d_{GH}(B_{\lambda_i^2\tilde{g}}(\tilde{x}_i,2), B_{\mathbb{R}^{2n}}(0,2)) < \delta
\end{align}
and $\op{rank}(\Gamma_\delta(\tilde{x}_i)) = 0$. Applying Theorem \ref{theorem simple version of Naber-Zhang}, we have
\begin{align}
|\op{Rm}_{\lambda_i^2g}(x_i)| = |\op{Rm}_{\lambda_i^2\tilde{g}}(\tilde{x}_i)| \leq C(n).
\end{align}

In the other subcase, we assume that $u(x_i) < n_i^{-1} \rightarrow 0$, so the function $\xi_2$ converges to $1$ on the ball $B_{\lambda_i^2g}(x_i,2)$. It follows that the last term \eqref{last term of approx of lambda^2g} converges to $\frac{4}{n-1}\lambda_i^2(1-\xi_1(x_i))[(du)^2 + u^2(\frac{n-1}{2}\theta_2^\p)^2]$. Recall \eqref{theta1}-\eqref{theta2} that
\begin{align}
\theta_1 &= -\frac{2}{n-1}\eta_2,\\
\theta_2 &= -2\eta_1 + \frac{2}{n-1}\eta_2.
\end{align}
It follows that
\begin{align}
\theta_1 + \theta_2 &= -2\eta_1,\\
\frac{n-1}{2}\theta_2 &= -(n-1)\eta_1 + \eta_2.
\end{align}
So the dual base with respect to $\theta_1 + \theta_2, \frac{n-1}{2}\theta_2$ is
\begin{align}
K_{\theta_1 + \theta_2} &= -\frac{1}{2}T_1 - \frac{n-1}{2}T_2,\\
K_{\frac{n-1}{2}\theta_2} &= T_2.
\end{align}
Since $T_2$ is the primitive generator of an $\mathbb{S}^1$-action, it follows that the cone angle of $(du)^2 + u^2(\frac{n-1}{2}\theta_2)^2$ is $2\pi$. Since the $\mathbb{S}^1$-orbit generated by $K_{\theta_1 + \theta_2}$ intersects with the $\mathbb{S}^1$-orbit generated by $K_{\frac{n-1}{2}\theta_2}$ at exactly $n-1$ points, the metric $\lambda_i^2g$ is $C^0$-close to $\mathbb{R}^{2n-3}\times((\mathbb{R}^2\times\mathbb{S}^1)/\mathbb{Z}_{n-1})$ with a length of $\mathbb{S}^1$ bounded by $\lambda_i\rightarrow 0$. Let $\tilde{g}$ be the local pull back of $g$ to $\mathbb{R}^{2n-3}\times \mathbb{R}^2\times \mathbb{S}^1$ and let $\tilde{x}_i$ be any pull back of $x_i$, then for sufficiently large $i$, we have
\begin{align}
d_{GH}(B_{\lambda_i^2\tilde{g}}(\tilde{x}_i,2), B_{\mathbb{R}^{2n-1}}(0,2)) < \delta
\end{align}
and $\Gamma_\delta(\tilde{x}_i) = \mathbb{Z}$. In particular $\op{rank}(\Gamma_\delta(\tilde{x}_i)) = 1$.

Applying Theorem \ref{theorem simple version of Naber-Zhang}, we have
\begin{align}
|\op{Rm}_{\lambda_i^2g}(x_i)| = |\op{Rm}_{\lambda_i^2\tilde{g}}(\tilde{x}_i)| \leq C(n).
\end{align}

But on the other hand, we have
\begin{align}
|\op{Rm}_{\lambda_i^2g}(x_i)| = \frac{1}{\lambda_i^2}|\op{Rm}_{g}(x_i)| \geq \frac{C_i}{\lambda_i^2 + \lambda_i^2\rho_i} = \frac{C_i}{\lambda_i^2 + m_i^2}\rightarrow +\infty.
\end{align}
Thus, we get a contradiction, finishing the proof of the proposition.
\end{proof}

As a consequence of the estimate of Riemannian curvature, applying \cite[Lemma 4.3]{Heinthesis} we have
\begin{proposition}\label{proposition quasi-atlas Apostolov}
For any $0<\alpha < 1$, the metric $(g,\omega)$ of Apostolov-Cifarelli admits a quasi-atlas which is $C^{k,\alpha}$ for any $k\geq 1$.
\end{proposition}

Here we recall the definition of quasi-atlas.
\begin{definition}
Let $(N,\omega_0,g_0)$ be a complete K\"ahler manifold. A $C^{k,\alpha}$ quasi-atlas for $(N,\omega_0,g_0)$ is a collection $\{\Phi_x|x\in A\}$, $A\subset N$, of holomorphic local
diffeomorphisms $\Phi_x: B\rightarrow N$, $\Phi_x(0)=x$, from $B=B(0,1)\subset \mathbb{C}^m$ into $N$ which extend smoothly to the closure $\bar{B}$, and such that there exists
$C\geq 1$ with $\op{inj}_{\Phi_x^*g_0}\geq \frac{1}{C}$, $\frac{1}{C}g_{\mathbb{C}^m}\leq \Phi_x^*g_0\leq Cg_{\mathbb{C}^m}$, and
$||\Phi_x^*g_0||_{C^{k,\alpha}(B,g_{\mathbb{C}^m})}\leq C$ for all $x\in A,$ and such that for all $y\in N$ there exists $x\in A$ with $y\in \Phi_x(B)$ and
$d_{g_0}(y,\partial\Phi_x(B))\geq \frac{1}{C}$.
\end{definition}
Given a $C^{k,\alpha}$ quasi-atlas, we can define global H\"older spaces of functions by setting
\begin{align}
    ||u||_{C^{k,\alpha}(N)} = \sup\{||u\circ \Phi_x||_{C^{k,\alpha}(B)}|x\in A\}.
\end{align}

Combining Proposition \ref{proposition volume growth of g}, Theorem \ref{theorem asymptotic cone of g} and Proposition \ref{proposition better estimation of Rm of g}, we conclude that
\begin{proposition}\label{proposition ALF of apo-cifa}
The metric $(\mathbb{C}^n, g)$ of Apostolov-Cifarelli is an ALF metric in the following sense:
\begin{itemize}
  \item The volume growth of $g$ is of order $2n-1$;
  \item The asymptotic cone of $g$ is a $(2n-1)$-dimensional metric cone;
  \item The sectional curvature of $g$ is bounded by $\frac{C}{\rho}$ for some $C>0$.
\end{itemize}
\end{proposition}

\section{ALF Calabi-Yau metrics modeled on the metric of Apostolov-Cifarelli}\label{section crepant resolution}
Following the previous section, let $(g,\omega)$ denote the metric of the Taub-NUT type of Apostolov and Cifarelli when $l=2$ and $a=0$. In this section, we are mainly interested in the case $n>2$, since the case $n=2$ corresponding to the Taub-NUT metric is well studied.

Consider the action of cyclic group $\mathbb{Z}_n$ of order $n$ on $\mathbb{C}^n$ generated by $(z_1,\dots,z_n)\mapsto (e^{\frac{2\pi i}{n}}z_1,\dots, e^{\frac{2\pi i}{n}}z_n)$. It is known that there is a crepant resolution $\pi: \mathcal{K}_{\mathbb{CP}^{n-1}}\rightarrow \mathbb{C}^n/\mathbb{Z}_n$, where $\mathcal{K}_{\mathbb{CP}^{n-1}}$ is the total space of the canonical bundle of $\mathbb{C}^{n-1}$. Then it is natural to ask whether there exists an ALF Calabi-Yau metric on $\mathcal{K}_{\mathbb{CP}^{n-1}}$ asymptotic to the quotient by $\mathbb{Z}_n$ of $(g,\omega)$.

More generally, by Theorem \cite[Theorem 1.4]{apostolov2023hamiltonian}, the metric of Apostolov-Cifarelli $(g,\omega)$ is invariant by the action of $U(1)\times U(n-1)$. Here the action of $U(1)$ comes from the rotation of the trivial bundle $\mathbb{C}$, and the action of $U(n-1)$ comes from its action on $O_{\mathbb{CP}^{n-2}}(-1)$. Assume that $\Gamma\subset U(1)\times U(n-1)$ is a finite subgroup such that the singularity $\mathbb{C}^n/\Gamma$ admits a crepant resolution $\pi: Y \rightarrow \mathbb{C}^n/\Gamma$. We will prove the following theorem in this section using the approach of Tian-Yau's work \cite{zbMATH04186562, zbMATH00059791} and result of of Hein \cite{Heinthesis}.
\begin{theorem}\label{theorem ALF metric asymptotic to apostolov}
For any compactly supported K\"ahler class of $Y$ and any $c>0$, there exists an ALF Calabi-Yau metric $\omega^\p$ having the same cohomology class on $Y$ which is asymptotic to $c\omega$ near the infinity. More precisely, we have
\begin{align}
|\nabla^k(\omega^\p - c\pi^*\omega)|_{\omega^\p} \leq C(k,\epsilon)(1 + \rho^\p)^{-2n+3+\epsilon},
\end{align}
where $\epsilon > 0$ is any small constant, $\rho^\p$ is the distance function from a point of $Y$ measured by $\omega^\p$ and $k\geq 0$.
\end{theorem}

\subsection{Construction of an asymptotic Calabi-Yau metric}
In this subsection, let $(Y, J)$ be a complex manifold of complex dimension $n$ ($n>2$) such that the canonical bundle $\mathcal{K}_Y$ is trivial, and assume that there is a plurisubharmonic function $K$. Denote $\omega=i\partial\bar{\partial}K$, we assume that $\omega$ is strictly positive outside a compact set. Let $\rho$ be a distance function measured by a complete Riemannian metric $g$ which coincides with the K\"ahler metric of $\omega$ outside a compact set. We assume that $K> 0$, $K$ is comparable to $\rho^2$ outside a compact set and $dK\wedge d^cK \leq C K dd^c K$, that is to say, the function $K$ satisfies all the three properties listed in Proposition \ref{proposition properties of potential}.

\begin{lemma}\label{lemma dd^cK^alpha>0}
There exists $0<\alpha_0<1$ such that for any $\alpha > \alpha_0$, we have $dd^c(K)^\alpha > 0$.
\end{lemma}
\begin{proof}
By the assumption of $K$, we have
\begin{align}\label{equation dd^cK^alpha}
dd^cK^\alpha = \alpha K^{\alpha-2}\left[(\alpha - 1)dK\wedge d^cK + Kdd^cK \right] > 0
\end{align}
by choosing $\alpha_0$ sufficiently close to $1$.
\end{proof}

\begin{lemma}\label{lemma h_alpha, general}
For $\alpha > \alpha_0$ where $\alpha_0$ is defined in Lemma \ref{lemma dd^cK^alpha>0}, there exists a strictly positive and smooth plurisubharmonic function $h_\alpha$ on $Y$ that is strictly plurisubharmonic and equal to $K^\alpha$ outside a compact set.
\end{lemma}
\begin{proof}
Let $\psi:\mathbb{R}\rightarrow\mathbb{R}$ be a smooth function such that $\psi^\p,\psi^{\p\p}\geq 0$, $\psi(t) = 3$ if $t<2$ and $\psi(t) = t$ if $t > 4$. Define $h_\alpha = \psi(K^\alpha)$, then the lemma follows from Lemma \ref{lemma dd^cK^alpha>0} and the following formula:
\begin{align}\label{dd^c of composition}
dd^c(\psi\circ f) = \psi^{\p\p}(f) df\wedge d^cf + \psi^\p(f)dd^cf.
\end{align}
\end{proof}

\begin{lemma}\label{lemma construction of asymptotic CY metric}
For any compactly supported K\"ahler class $\omega_Y\in H^2_c(Y, \mathbb{R})$ and any $c>0$, there exists a K\"ahler form $\hat{\omega}$ on $Y$ having the same K\"ahler class as $\omega_Y$ such that $\hat{\omega} = c\omega$ outside a compact set.
\end{lemma}
\begin{proof}
Since $[\omega_Y]\in H^2_c(Y,\mathbb{R})$, by \cite[Corollary A.3]{Conlon-Hein1}, there exists a smooth real function $v$ such that $\omega_Y = -i\partial\bar\partial v$ when $K > R$ for $R$ sufficiently large.

Fix $\alpha \in (\alpha_0, 1)$, we can assume (by enlarging $R$ if necessary) that $h_\alpha = K^\alpha$ and $h_1 = K$ when $K>R$ where $h_\alpha,h_1$ are defined in Lemma \ref{lemma h_alpha, general}.

Fix a cutoff function $\chi: \mathbb{R}\rightarrow [0,1]$ satisfying $\chi(s) = 0$ if $s<2R$ and $\chi(s) = 1$ if $s>3R$. Define $\zeta: Y \rightarrow \mathbb{R}$ by $\zeta(y) = \chi(K(y))$. For $S>2$, define $\zeta_S(y) = \chi(\frac{K(y)}{S})$. Note that $0<2R<3R<2SR<3SR<+\infty$.

Given $c>0$, we construct
\begin{align}
\hat{\omega} = \omega_Y + i\partial\bar\partial(\zeta v) + C i\partial\bar\partial((1-\zeta_S)h_\alpha) + ci\partial\bar\partial h_1,
\end{align}
with $C$ and $S$ to be determined. It is clear that $\hat{\omega}$ lies in the same cohomology class as $\omega_Y$.

On the region $\{K < 2R \}$, $\hat{\omega} = \omega_Y + Ci\partial\bar\partial h_\alpha + ci\partial\bar\partial h_1 \geq \omega_Y > 0$.

On the region $\{3R < K < 2SR\}$, $\hat{\omega} = Ci\partial\bar\partial h_\alpha + ci\partial\bar\partial h_1 > 0$.

On the region $\{3SR < K\}$, $\hat{\omega} = ci\partial\bar\partial h_1 = c\omega > 0$.

On the region $\{2R \leq K \leq 3R \}$, $\hat{\omega} = \omega_Y + i\partial\bar\partial(\zeta v) + Ci\partial\bar\partial h_\alpha + ci\partial\bar\partial h_1 > 0$ if $C$ is made large enough.

Finally on the region $\{2SR\leq K \leq 3SR \}$, $\hat{\omega} = C i\partial\bar\partial((1-\zeta_S)h_\alpha) + ci\partial\bar\partial h_1$. By assumption of $K$, we know that
\begin{align}
|dK|_{i\partial\bar\partial K}, |d^cK|_{i\partial\bar\partial K} \leq C^\p(K)^{\frac{1}{2}}.
\end{align}
After some simple derivation, it follows that $|i\partial\bar\partial((1-\zeta_S)h_\alpha)|_{i\partial\bar\partial h_1} \leq C^{\p\p}S^{-(1-\alpha)}$ on this region, so $\hat{\omega} > 0$ if $S$ is made large enough.
\end{proof}

\begin{lemma}\label{lemma weight function}
The smooth function $h_\frac{1}{2}$ defined in Lemma \ref{lemma h_alpha, general} is comparable to $1 + \rho$, and $|\nabla h_\frac{1}{2}| + h_\frac{1}{2}|dd^c h_\frac{1}{2}|$ is bounded on $Y$. Here the norm and Laplacian are calculated with respect to $g$.
\end{lemma}
\begin{proof}
That $h_\frac{1}{2}$ is comparable to $1 + \rho$ is a consequence of the assumption that $K$ is comparable to $\rho^2$.

For the proof of the boundedness of $|\nabla h_\frac{1}{2}| + h_\frac{1}{2}|dd^c h_\frac{1}{2}|$, it suffices to prove the boundedness of the same formula replacing $h_\frac{1}{2}$ by $w=K^\frac{1}{2}$ outside of a compact set.

First we calculate $dw = \frac{1}{2}K^{-\frac{1}{2}}dK$, so by the assumption that $dK\wedge d^cK \leq C K dd^c K$ we have $|dw|\leq C$.
Next for $dd^cw$ we note that
\begin{align}
dd^cK^\frac{1}{2} = \frac{1}{2}(K)^{-\frac{3}{2}}\left[-\frac{1}{2}dK\wedge d^cK + Kdd^cK \right].
\end{align}
So the boundedness of $w|dd^c w|$ also follows from the assumption that $dK\wedge d^cK \leq C K dd^c K$ and \eqref{equation dd^cK^alpha}.
\end{proof}

\subsection{Proof of Theorem \ref{theorem ALF metric asymptotic to apostolov}}
Instead of proving Theorem \ref{theorem ALF metric asymptotic to apostolov} directly, we will prove a more general proposition with the help of the following result of \cite{Heinthesis}.

\begin{proposition}\label{hein's thesis}
Let $(N,\omega_0,g_0)$ be a complete noncompact K\"ahler manifold of complex dimension $m$ with a $C^{3,\alpha}$ quasi-atlas which satisfies $SOB(\beta)$ for some $\beta>2$. Let
$f\in C^{2,\alpha}(N)$ satisfies $|f|\leq Cr^{-\mu}$ on $\{r>1\}$ for some $\beta>\mu>2$. Then there exist
$\bar{\alpha}\in (0,\alpha]$ and $u\in C^{4,\bar{\alpha}}$ such that $(\omega_0+i\partial\bar\partial u)^m=e^f\omega_0^m$ and that $\omega_0+i\partial\bar\partial u$ is a K\"ahler form uniformly equivalent to $\omega_0$. If in addition $f\in C^{k,\bar\alpha}_{loc}(N)$ for some $k\geq 3$, then all such solutions $u$ belong to
$C^{k+2,\bar\alpha}_{loc}(N)$.

Moreover, if there is a function $\tilde\rho$ on $N$ comparable to $1+d_{g_0}(x_0,-)$ for some $x_0\in N$, and $\tilde\rho$ satisfies $|\nabla\tilde\rho| +
\tilde\rho|dd^c\tilde\rho|\leq C$ for some $C>0$, then we have the decay estimate $|u|\leq C(\epsilon)r^{2-\mu+\epsilon}$ for any sufficiently small $\epsilon>0$.
\end{proposition}

\begin{proposition}\label{proposition CY metric on resolution}
Let $(M, g_M, \omega_M)$ be a complete Calabi-Yau metric of complex dimension $n$ with $n>2$. Assume that $M$ admits a K\"ahler potential $K_M$ such that $K_M> 0$, $K_M$ is comparable to $\rho_M^2$ outside a compact set (where $\rho_M$ is a distance function of $M$ measured by $g_M$), and $dK_M\wedge d^cK_M \leq C K_Mdd^c K_M$ for some $C>0$. We also assume that $M$ admits a quasi-atlas which is $C^{k,\alpha}$ for any $k\geq 1$, and $(M,g_M)$ satisfies the $SOB(\beta)$ property for some $\beta>2$. Suppose that there is a finite group $\Gamma$ acting on $M$ preserving $(g_M,\omega_M)$, and the quotient $M/\Gamma$ admits a crepant resolution $\pi: Y \rightarrow M/\Gamma$. Then for any compactly supported K\"ahler class $[\omega_Y]$ of $Y$ and any $c>0$, there exists a Calabi-Yau metric $\omega^\p$ having the same cohomology class as $\omega_Y$ which is asymptotic to $c\pi^*\omega_M$ near the infinity. More precisely we have
\begin{align}
|\nabla^k(\omega^\p - c\pi^*\omega_M)|_{\omega^\p} \leq C(k,\epsilon)(1 + \rho^\p)^{-\beta + 2 +\epsilon},
\end{align}
where $\epsilon > 0$, and $\rho^\p$ is the distance function from a point of $Y$ measured by $\omega^\p$ and $k\geq 0$.
\end{proposition}
\begin{proof}
First, without loss of generality, we may assume that $K_M$ is invariant by $\Gamma$. If not, let $\tilde{K}_M = \frac{1}{|\Gamma|}\sum_{\gamma\in\Gamma}\gamma^*K_M$ be the average of $K$ by $\Gamma$, then $\tilde{K}_M$ is still a positive K\"ahler potential comparable to $\rho_M^2$. Note that $dK\wedge d^cK \leq CKdd^cK$ is equivalent to $(dK)^2 \leq C K g$, so by enlarging $C$ if necessary, we still have $d\tilde{K}_M\wedge d^c\tilde{K}_M \leq C \tilde{K}_Mdd^c \tilde{K}_M$. So we can assume that $K_M$ is well-defined on $M/\Gamma$.

Applying Lemma \ref{lemma construction of asymptotic CY metric} with $K = \pi^*K_M$, there exists a K\"ahler form $\hat{\omega}$ on $Y$ having the same cohomology class as $\omega_Y$ and coincides with $c\pi^*\omega_M$ outside a compact set. Let $\Omega_Y$ be a holomorphic volume form on $Y$ such that $i^{n^2}\Omega_Y\wedge \bar\Omega_Y = \pi^*\omega_M^n$ and let $\hat{f}$ be the Ricci potential
\begin{align}
\hat{f} = \log \frac{i^{n^2}\Omega_Y\wedge \bar\Omega_Y}{(\hat{\omega}/c)^n}.
\end{align}
Then $\hat{f}$ is compactly supported on $Y$ and in particular, we have $|\hat{f}| \leq C\rho_{\hat{\omega}}^{-\mu}$ for any $2<\mu<2n-1$, where $\rho_{\hat{\omega}}$ is a distance function measured by $\hat{\omega}$. By Proposition \ref{hein's thesis}, there exists $\bar{\alpha\in (0, \alpha]}$ such that we get a solution $u\in C^{4,\bar{\alpha}}(Y)$ of the Monge-Amp\`{e}re equation
\begin{align}
(\hat{\omega} + i\partial\bar\partial u)^n = e^{\hat{f}}\hat{\omega}^n.
\end{align}
Let $\omega^\p = \hat{\omega} + i\partial\bar\partial u$, then $\omega^\p$ is Calabi-Yau and it is uniformly equivalent to $\hat{\omega}$ hence $\rho^\p$ is comparable to $\rho_{\hat{\omega}}$. By Lemma \ref{lemma weight function}, we may apply the second part of Proposition \ref{hein's thesis} to have $|u| \leq C(\epsilon)(\rho^\p)^{-\beta + 2 +\epsilon}$ for any $\epsilon > 0$ sufficiently small.

If we think of $\omega^\p,\hat{\omega}$ as given, then the Monge-Amp\`{e}re equation can be written as
\begin{align}
(e^{\hat{f}} - 1)\hat{\omega}^n = i\partial\bar\partial u \wedge \sum_{k=1}^{n-1}(\omega^\p)^k\wedge \hat\omega^{n-1-k},
\end{align}
and it can be viewed as an elliptic equation of $u$. Outside a compact set, the left hand side is zero so by Schauder estimates on each chart of the quasi-atlas outside this compact set, we find that $|\nabla^ku|_{\omega^\p} \leq C(k,\epsilon)(\rho^\p)^{-\beta + 2 +\epsilon}$ for $k\geq 0$.
\end{proof}

\begin{proof}[Proof of Theorem \ref{theorem ALF metric asymptotic to apostolov}]
We wish to apply Proposition \ref{proposition CY metric on resolution} to $\beta=2n-1$ and the metric of Apostolov-Cifarelli $(\mathbb{C}^n, g, \omega)$, which is complete and Calabi-Yau. By Proposition \ref{proposition properties of potential}, it admits a K\"ahler potential $H$ such that $H\geq 0$, $H$ is comparable to $\rho^2$ outside a compact set and $dH \wedge d^cH \leq C^\p Hdd^cH$ for some $C^\p > 0$. The existence of quasi-atlas is proved in Proposition \ref{proposition quasi-atlas Apostolov}, and the $SOB(2n-1)$ property of $g$ is proved in Proposition \ref{proposition SOB(2n-1)}. Thus all the assumptions of Proposition \ref{proposition CY metric on resolution} are satisfied, and it produces the ALF Calabi-Yau metric on the crepant resolution asymptotic to the metric of Apostolov-Cifarelli.
\end{proof}

\subsection{ALF Calabi-Yau metrics on $\mathcal{K}_{\mathbb{CP}^{n-1}}$ modeled on the metric of Apostolov-Cifarelli}
Recall that the action of cyclic group $\mathbb{Z}_n$ of order $n$ on $\mathbb{C}^n$ generated by $(z_1,\dots,z_n)\mapsto (e^{\frac{2\pi i}{n}}z_1,\dots, e^{\frac{2\pi i}{n}}z_n)$ is a subgroup action of the standard action of $\mathbb{T}^n$ on $\mathbb{C}^n$. Since the metric $(\mathbb{C}^n, g, \omega)$ of Apostolov-Cifarelli is $\mathbb{T}^n$-equivariantly biholomorphic to the standard $\mathbb{C}^n$, we know that $\mathbb{Z}_n$ acts on $(\mathbb{C}^n, g, \omega)$ holomorphically. Moreover, the K\"ahler potential $H$ depends only on $\xi_1,\xi_2$, and $\xi_1,\xi_2$ only depends on the moment maps $\sigma_1,\sigma_2$, so the potential $H$ is invariant by $\mathbb{T}^n$, hence $\mathbb{Z}_n$. Applying Theorem \ref{theorem ALF metric asymptotic to apostolov}, we get
\begin{proposition}\label{proposition ALF CY metrics on K_CP apostolov}
For any $c>0$ and $\epsilon>0$, there exists an ALF Calabi-Yau metric $\omega^\p$ on $\mathcal{K}_{\mathbb{CP}^{n-1}}$ in the sense that $\omega^\p$ is asymptotic to $c\omega$ and $[\omega^\p] = \epsilon[\omega_{Cal}]$, where $\omega_{Cal}$ is the Calabi metric constructed in \cite{calabi1979metriques}.
\end{proposition}
It is then interesting to figure out the asymptotic cone of $\omega^\p$ constructed in Proposition \ref{proposition ALF CY metrics on K_CP apostolov}.
\begin{proposition}
The asymptotic cone of $\omega^\p$ in Proposition \ref{proposition ALF CY metrics on K_CP apostolov} is $(\mathbb{C}^{n-1}/\mathbb{Z}_{k(n-1)})\times \mathbb{R}$, where $k = n$ if $n$ is odd and $k=\frac{n}{2}$ if $n$ is even. Here $\mathbb{Z}_{k(n-1)}$ acts on $\mathbb{C}^{n-1}$ by multiplying $\zeta_{k(n-1)} = e^{\frac{2\pi i}{k(n-1)}}$.
\end{proposition}
\begin{proof}
Note that this $\mathbb{Z}_n$-action is a subgroup action of the $\mathbb{T}^2$-action, in fact its generator corresponds to $\frac{1}{n}(T_1 + T_2)$ under $\exp\circ2\pi$. So it suffices to understand the $\mathbb{T}^2$-action on $(\mathbb{C}^{n-1}/\mathbb{Z}_{n-1})\times\mathbb{R}$ induced by the Gromov-Hausdorff approximation $f_\lambda: ((\mathbb{C}^{n-1}\times \mathbb{S}^1)/\mathbb{Z}_{n-1})\times \mathbb{R} \rightarrow (\mathbb{C}^{n-1}/\mathbb{Z}_{n-1})\times\mathbb{R}$. Since all the $f_\lambda$ is obtained from $f_1$ by a scaling, it suffices to understand the $\mathbb{T}^2$-action on $(\mathbb{C}^{n-1}/\mathbb{Z}_{n-1})\times\mathbb{R}$ induced by $f_1$. By its definition, the map of $f_1$ maps the $\mathbb{S}^1$-orbit generated by $K_1$ to a point. Thus the $\mathbb{T}^2$-action on $(\mathbb{C}^{n-1}/\mathbb{Z}_{n-1})\times\mathbb{R}$ degenerates to the $\mathbb{S}^1 = \mathbb{T}^2/Span\{e_1\}$-action generated by $K_2$.

Recall that $\mathbb{T}^2 = \mathbb{R}^2/\Gamma_v$ where $\Gamma_v$ is the lattice generated by $v_1 = (0,-2)$, $v_2 = \frac{2}{n-1}(-1,1)$, and $v_i$ corresponds to the vector field $T_i$. Then $e_1=(1,0)$ corresponds to $K_1$ and $e_2 = (0,1)$ corresponds to $K_2$, and $\theta_i(K_j) = \delta_{ij}$. It follows that the subgroup in $\mathbb{T}^2$ generated by the direction of $e_1$ intersects with the subgroup generated by $e_2$ at exactly $n-1$ points, which explains why $\Lambda = \mathbb{Z}_{n-1}$. In fact, taking the quotient of $\mathbb{T}^2$ by the subgroup generated by $e_1$, we get $\mathbb{T}^2/Span\{e_1\} = \mathbb{R}/\frac{2}{n-1}\mathbb{Z}$ and the quotient map $\mathbb{T}^2\rightarrow \mathbb{T}^2/Span\{e_1\}$ is simply the map taking the second coordinate.

Now $\frac{1}{n}(v_1+v_2) = (\frac{2}{n(n-1)}, \frac{-2(n-2)}{n(n-1)})$, so its image under the quotient map is $\frac{-2(n-2)}{n(n-1)}$ modulo $\frac{2}{n-1}\mathbb{Z}$. So it generates a subgroup $\mathbb{Z}_k$ of $\mathbb{S}^1=\mathbb{T}^2/Span\{e_1\}$ acting on $(\mathbb{C}^{n-1}/\mathbb{Z}_{n-1})\times\mathbb{R}$, where $k=n$ if $n$ is odd and $k=\frac{n}{2}$ if $n$ is even. Thus the asymptotic cone of $\omega^\p$ in Proposition \ref{proposition ALF CY metrics on K_CP apostolov} is $(\mathbb{C}^{n-1}/\mathbb{Z}_{k(n-1)})\times \mathbb{R}$.
\end{proof}

\bibliographystyle{abbrv}
\bibliography{bibnoteApostolovCifarelli}

\begin{thebibliography}{10}

\bibitem{ACGI}
V.~Apostolov, D.~M.~J. Calderbank, and P.~Gauduchon.
\newblock Hamiltonian 2-forms in {K{\"a}hler} geometry. {I}: {General} theory.
\newblock {\em J. Differ. Geom.}, 73(3):359--412, 2006.

\bibitem{apostolov2023hamiltonian}
V.~Apostolov and C.~Cifarelli.
\newblock Hamiltonian $2$-forms and new explicit {C}alabi--{Y}au metrics and
  gradient steady {K}\"ahler--{R}icci solitons on $\mathbb{C}^n$.
\newblock {\em J. Differ. Geom.}, forthcoming.

\bibitem{biquard2011kummer}
O.~Biquard and V.~Minerbe.
\newblock A {Kummer} construction for gravitational instantons.
\newblock {\em Commun. Math. Phys.}, 308(3):773--794, 2011.

\bibitem{buser1982}
P.~Buser.
\newblock A note on the isoperimetric constant.
\newblock {\em Ann. Sci. {\'E}c. Norm. Sup{\'e}r. (4)}, 15:213--230, 1982.

\bibitem{calabi1979metriques}
E.~Calabi.
\newblock M{\'e}triques k{\"a}hleriennes et fibres holomorphes.
\newblock {\em Ann. Sci. {\'E}c. Norm. Sup{\'e}r. (4)}, 12:269--294, 1978.

\bibitem{cifarelli2024explicitcompletericciflatmetrics}
C.~Cifarelli.
\newblock Explicit complete ricci-flat metrics and k\"{a}hler-ricci solitons on
  direct sum bundles.
\newblock {\em arXiv:2410.23645 [math.DG]}, 2024.

\bibitem{Colding-Naber}
T.~H. Colding and A.~Naber.
\newblock Sharp {H{\"o}lder} continuity of tangent cones for spaces with a
  lower {Ricci} curvature bound and applications.
\newblock {\em Ann. Math. (2)}, 176(2):1173--1229, 2012.

\bibitem{ConlonRochon2019}
R.~J. Conlon, A.~Degeratu, and F.~Rochon.
\newblock Quasi-asymptotically conical {Calabi}-{Yau} manifolds.
\newblock {\em Geom. Topol.}, 23(1):29--100, 2019.

\bibitem{Conlon-Hein1}
R.~J. Conlon and H.-J. Hein.
\newblock Asymptotically conical {C}alabi-{Y}au manifolds, {I}.
\newblock {\em Duke Math. J.}, 162(15):2855--2902, 2013.

\bibitem{ConlonRochon2021}
R.~J. Conlon and F.~Rochon.
\newblock New examples of complete {Calabi}-{Yau} metrics on
  {{\(\mathbb{C}^n\)}} for {{\(n\ge 3\)}}.
\newblock {\em Ann. Sci. {\'E}c. Norm. Sup{\'e}r. (4)}, 54(2):259--303, 2021.

\bibitem{conlon2023warped}
R.~J. Conlon and F.~Rochon.
\newblock Warped quasi-asymptotically conical calabi-yau metrics, 2023.

\bibitem{Heinthesis}
H.-J. Hein.
\newblock {\em On gravitational instantons}.
\newblock PhD thesis, Princeton University, 2010.

\bibitem{joyce2000compact}
D.~D. Joyce.
\newblock {\em Compact manifolds with special holonomy}.
\newblock Oxford Math. Monogr. Oxford: Oxford University Press, 2000.

\bibitem{Yangli2017}
Y.~Li.
\newblock A new complete {Calabi}-{Yau} metric on {{\({\mathbb {C}}^3\)}}.
\newblock {\em Invent. Math.}, 217(1):1--34, 2019.

\bibitem{min2023construction}
D.~Min.
\newblock Construction of higher dimensional {ALF} {C}alabi-{Y}au metrics.
\newblock {\em Ann. Sci. {\'E}c. Norm. Sup{\'e}r.}, forthcoming.

\bibitem{Minerbe2009}
V.~Minerbe.
\newblock Weighted {Sobolev} inequalities and {Ricci} flat manifolds.
\newblock {\em Geom. Funct. Anal.}, 18(5):1696--1749, 2009.

\bibitem{Naber-Zhang}
A.~Naber and R.~Zhang.
\newblock Topology and {{\(\epsilon\)}}-regularity theorems on collapsed
  manifolds with {Ricci} curvature bounds.
\newblock {\em Geom. Topol.}, 20(5):2575--2664, 2016.

\bibitem{zbMATH07698520}
S.~Sun and J.~Zhang.
\newblock No semistability at infinity for {Calabi}-{Yau} metrics asymptotic to
  cones.
\newblock {\em Invent. Math.}, 233(1):461--494, 2023.

\bibitem{zbMATH07131296}
G.~Sz{\'e}kelyhidi.
\newblock Degenerations of {{\(\mathbf{C}^n\)}} and {Calabi}-{Yau} metrics.
\newblock {\em Duke Math. J.}, 168(14):2651--2700, 2019.

\bibitem{zbMATH05234298}
G.~Tian.
\newblock Aspects of metric geometry of four manifolds.
\newblock In {\em Inspired by S. S. Chern. A memorial volume in honor of a
  great mathematician}, pages 381--397. Hackensack, NJ: World Scientific, 2006.

\bibitem{zbMATH04186562}
G.~Tian and S.-T. Yau.
\newblock Complete {K{\"a}hler} manifolds with zero {Ricci} curvature. {I}.
\newblock {\em J. Am. Math. Soc.}, 3(3):579--609, 1990.

\bibitem{zbMATH00059791}
G.~Tian and S.-T. Yau.
\newblock Complete {K{\"a}hler} manifolds with zero {Ricci} curvature. {II}.
\newblock {\em Invent. Math.}, 106(1):27--60, 1991.

\bibitem{VanCoevering}
C.~van Coevering.
\newblock Ricci-flat {K{\"a}hler} metrics on crepant resolutions of
  {K{\"a}hler} cones.
\newblock {\em Math. Ann.}, 347(3):581--611, 2010.

\end{thebibliography}
\end{document}